\documentclass{amsart}

\usepackage{placeins}

\usepackage{hyphenat}
\usepackage{amssymb, amsmath, amsthm, amsfonts, mathtools}

\usepackage{mathabx}

\allowdisplaybreaks

\usepackage{hyperref}
\usepackage[nameinlink]{cleveref}
\usepackage{autonum} 

\usepackage{xcolor}

\usepackage{tensor}

\usepackage{wasysym}

\usepackage[normalem]{ulem}

\usepackage{pdfpages}

\usepackage{graphicx}

\usepackage{marginnote}

\usepackage{xfrac}
\usepackage{nicefrac}
\usepackage{faktor}

\usepackage{cancel}

\usepackage{slashed}

\usepackage{tikz}
\usetikzlibrary{matrix,calc,arrows.meta}
\usepackage{tikz-cd}

\usepackage{comment}

\usepackage{esvect}

\usepackage{braket}

\usepackage{enumitem}

\usepackage{nicematrix}

\usepackage{thmtools}

\theoremstyle{plain}
\newtheorem{theorem}{Theorem}[section]
\crefname{theorem}{Theorem}{Theorems}
\newtheorem{proposition}[theorem]{Proposition}
\crefname{proposition}{Proposition}{Propositions}
\newtheorem{lemma}[theorem]{Lemma}
\crefname{lemma}{Lemma}{Lemmas}
\newtheorem{corollary}[theorem]{Corollary}
\crefname{corollary}{Corollary}{Corollaries}
\theoremstyle{definition}
\newtheorem{definition}[theorem]{Definition}
\crefname{definition}{Definition}{Definitions}
\newtheorem{example}[theorem]{Example}
\crefname{example}{Example}{Examples}
\newtheorem{def-prop}[theorem]{Definition-Proposition}
\crefname{def-prop}{Definition-Proposition}{}
\theoremstyle{remark}
\newtheorem{remark}[theorem]{Remark}
\crefname{remark}{Remark}{Remarks}

\crefname{claim}{Claim}{Claims}

\crefname{notation}{Notation}{}

\crefname{question}{Question}{Questions}

\crefname{chapter}{Chapter}{Chapters}
\crefname{section}{Section}{Sections}
\crefname{subsection}{Subsection}{Subsections}
\crefname{figure}{Figure}{Figures}

\usepackage[figurename=Figure,tablename=Table]{caption}

\newcommand{\C}{\mathbb{C}}
\newcommand{\Z}{\mathbb{Z}}

\newcommand{\overbar}[1]{\mkern 1.5mu\overline{\mkern-1.5mu#1\mkern-1.5mu}\mkern 1.5mu}

\newcommand*\colvec[3][]{
	\begin{pmatrix}\ifx\relax#1\relax\else#1\\\fi#2\\#3\end{pmatrix}
}

\DeclareMathOperator{\vsspan}{span}

\DeclareMathOperator{\Tr}{Tr}

\DeclareMathOperator{\id}{id}

\DeclareMathOperator{\co}{co}
\DeclareMathOperator{\ex}{ex}
\DeclareMathOperator{\prop}{prop}

\usepackage{mathbbol} %for \mathbb with upper case greek letters
\DeclareSymbolFontAlphabet{\mathbbm}{bbold}%
\DeclareSymbolFontAlphabet{\mathbb}{AMSb}%

\def\black{\color{black}}

\def\B{\mathcal{B}}
\def\cb{\textup{cb}}
\def\F{\mathcal{F}}
\def\H{\mathcal{H}}
\def\id{\text{id}}
\def\S{\mathcal{S}}
\def\T{\mathbb{T}}

\title[GH convergence of spectral truncations for tori]{Gromov--Hausdorff convergence of spectral truncations for tori}
\author{Malte Leimbach}

\author{Walter D. van Suijlekom}
\address{Institute for Mathematics, Astrophysics and Particle Physics, Radboud
	University Nijmegen, Heyendaalseweg 135, 6525 AJ Nijmegen, The Netherlands.
}
\email{m.leimbach@math.ru.nl, waltervs@math.ru.nl}
\date{January 17, 2024}

\begin{document}
	
	\begin{abstract}
		We consider operator systems associated to spectral truncations of tori.
		We show that their state spaces, when equipped with the Connes distance function, converge in the Gromov--Hausdorff sense to the space of all Borel probability measures on the torus equipped with the Monge--Kantorovich distance.
		A crucial role will be played by the relationship between Schur and Fourier multipliers.
		Along the way, we introduce the spectral Fejér kernel and show that it is a good kernel.
		This allows to make the estimates sufficient to prove the desired convergence of state spaces.
		We conclude with some structure analysis of the pertinent operator systems, including the C*-envelope and the propagation number, and with an observation about the dual operator system. 
	\end{abstract}
	
	\maketitle
	
	\tableofcontents
	
	\section{Introduction}
	In a spectral approach to geometry, such as the one advocated in noncommutative geometry \cite{Con94}, a natural question that arises is how one may approximate a geometric space if only part of the spectral data is available \cite{CvS21}. 
	In fact, for the so-called spectral truncations that are relevant in this paper one is naturally led to consider quantum metric spaces in the sense of Rieffel \cite{Rie04}.

	More precisely, we consider a {\em spectral triple} $(A,\mathcal H,D)$ consisting of a C*-algebra $A$ acting as bounded operators on a Hilbert space $\mathcal H$, together with a self-adjoint operator $D$ such that $[D,a]$ extends to a bounded operator for $a$ in a dense $*$-subalgebra of $A$, and such that the resolvent of $D$ is a compact operator. A spectral truncation for this set of data is given in terms of a spectral projection $P_\Lambda$ of $D$ onto the eigenspaces with eigenvalues of modulus $\leq \Lambda$. The C*-algebra $A$ is then replaced by an operator system, to wit $P_\Lambda A P_\Lambda$. It acts on the Hilbert space $P_\Lambda \mathcal H$ and also $D$ restricts to a self-adjoint operator thereon. 
	
	The above question may then be sharply formulated in terms of Gromov--Hausdorff convergence (as $\Lambda \to \infty$) of the state spaces $\S(P_\Lambda A P_\Lambda)$ when equipped with the Connes distance function \cite{Con94}
	\begin{align}\label{eqn:DistanceFormula-trunc}
		d_{P_\Lambda A P_\Lambda}(\varphi, \psi) &= \sup_{a \in P_\Lambda A P_\Lambda} \{|\varphi(a)-\psi(a)| \, : \, \|[P_\Lambda D P_\Lambda ,a]\| = 1\}.
		\intertext{The limit structure one is aiming for is then of course the state space $\S(A)$, this time equipped with the metric }
		d_A(\varphi, \psi) &= \sup_{a \in  A}  \{|\varphi(a)-\psi(a)| \, : \, \|[ D  ,a]\| = 1\}.\label{eqn:DistanceFormula}
	\end{align}
	Note that in the definition of quantum metric spaces there is an additional (and natural) requirement of compatibility betweeen the above metric structure and the given weak-$\ast$ topology on the pertinent state spaces ({\em \emph{cf.}} \cite[Definition 2.1]{Rie04}). If $A = \mathrm{C}(M)$ and $D=D_M$ is the Dirac operator on a compact Riemannian spin manifold $M$ ---indeed our case of interest--- this is the case, as was shown in \cite{Con89}. In fact, Connes'  distance function then coincides with the Monge--Kantorovich distance on the state space $\S(\mathrm{C}(M))$ of all Borel probability measures. In addition, for finite-dimensional operator systems ---such as the above $P_\Lambda A P_\Lambda$--- Connes' distance function metrizes the weak-$\ast$ topology as follows from \cite[Theorem 1.8]{Rie98} (see also \cite[Remark 6]{vS21}). 
	
	In a previous work, a convergence of spectral truncations has been established for the circle, when equipped with the natural Dirac operator \cite{vS21}. This was formulated in the following general framework, which we will also exploit here.

	\begin{definition} An {\em operator system spectral triple} is a triple $(\mathcal E,\mathcal H,D)$ where $\mathcal E$ is a dense subspace of a (concrete) operator system $E$ in $\mathcal B(\mathcal H)$, $\mathcal H$ is a Hilbert space and $D$ is a self-adjoint operator in $\H$ with compact resolvent and such that $[D,T]$ is a bounded operator for all $T \in \mathcal E$.
	\end{definition}
	The relation between sequences of operator system spectral triples and a limit structure is then captured by the following notion of an approximate order isomorphism. 
	\begin{definition}
		\label{defn:C1-approx}
		Let $(\mathcal{E}, \mathcal{H}, D)$ and $(\mathcal{E}_\Lambda, \mathcal{H}_\Lambda, D_\Lambda)$, for all $\Lambda \geq 0$, be operator system spectral triples.
		For all $\Lambda$, let $\rho_\Lambda : \mathcal{E} \rightarrow \mathcal{E}_\Lambda$ and $\sigma_\Lambda : \mathcal{E}_\Lambda \rightarrow \mathcal{E}$ be linear maps with the following properties:
		\begin{enumerate}[label=(\roman*)]
			\item $\rho_\Lambda$ and $\sigma_\Lambda$ are positive and unital,
			\item $\rho_\Lambda$ and $\sigma_\Lambda$ are \emph{$\mathrm{C}^1$-contractive}, i.e.\@ they are contractive both with respect to norm $\|\cdot\|$ and with respect to Lipschitz seminorm $\|[D,\cdot]\|$, respectively $\|[D_\Lambda,\cdot]\|$,
			\item the compositions $\sigma_\Lambda \circ \rho_\Lambda$ and $\rho_\Lambda \circ \sigma_\Lambda$ approximate the respective identities on $\mathcal{E}$ and $\mathcal{E}_\Lambda$ with respect to Lipschitz seminorm, i.e.
			\begin{align}
				\|a - \rho_\Lambda \circ \sigma_\Lambda(a)\| &\leq \gamma_\Lambda \|[D,a]\| \\
				\|T - \sigma_\Lambda \circ \rho_\Lambda(T)\| &\leq \gamma_\Lambda \|[D_\Lambda,T]\|,
			\end{align}
		\end{enumerate}
		for some net $(\gamma_\Lambda)_{\Lambda \geq 0}$ with $\gamma_\Lambda \rightarrow 0$ as $\Lambda \rightarrow \infty$. 
		Then we call the pair of maps $(\rho,\sigma)$ (by which we mean the collection of pairs of maps $(\rho_\Lambda,\sigma_\Lambda)$) a \emph{$\mathrm{C}^1$-approximate order isomorphism}.
	\end{definition}
	It was then shown in \cite[Theorem 5]{vS21} that if the metrics $d_{\mathcal{E}_\Lambda}$, for all $\Lambda\geq 0$, and $d_{\mathcal{E}}$ (defined as in \eqref{eqn:DistanceFormula-trunc} and \eqref{eqn:DistanceFormula}) metrize the respective weak$^*$-topologies on the state spaces $\mathcal{S}(\mathcal{E}_\Lambda)$ and $\mathcal{S}(\mathcal{E})$ and if a $\mathrm{C}^1$-approximate order isomorphism exists, then the sequence of metric spaces $(\mathcal{S}(\mathcal{E}_\Lambda), d_{\mathcal{E}_\Lambda})$ converges to $(\mathcal{S}(\mathcal{E}), d_{\mathcal{E}})$ in Gromov--Hausdorff distance.
	By exploiting this criterion it was shown in \cite{vS21} that spectral truncations of the circle converge.
	
	Other results on Gromov--Hausdorff convergence that can be cast in this general framework include \cite{Ber19,AKK22,Rie22}. Note also the recent developments around the related notion of propinquity \cite{Lat16,Lat16b,Lat18}, also in the context of spectral triples, but which, however, mainly focuses on C*-algebras. 
	It should be mentioned that, even though we are phrasing the question about convergence of spectral truncations in terms of Gromov--Hausdorff distance, another point of view would be quantum Gromov--Hausdorff distance \cite{Rie04}.
	In fact, these notions are not equivalent \cite{KK22b}.
	However, as the authors point out, it follows from \cite[Proposition 2.14]{KK22a} that Gromov--Hausdorff convergence still implies quantum Gromov--Hausdorff convergence in the case which we consider below.
	
	In this article, we show that the conditions in the above Definition can be met for tori in any dimension $d \geq 1$. Let us spend the remainder of this introduction by giving an extended overview of our setup and approach.
	
	\bigskip
	
	Consider the spectral triple of the $d$-dimensional torus $\mathbb{T}^d = \mathbb{R}^d/2\pi\mathbb{Z}^d$:
	\begin{align}%\label{eqn:SpectralTripleTorus}
		\left(\mathrm{C}^\infty(\mathbb{T}^d), \mathrm{L}^2(S(\mathbb{T}^d)), D\right)
	\end{align}
	This consists of the $*$-algebra of smooth functions acting on the Hilbert space of $\mathrm{L}^2$-sections of the spinor bundle $S(\mathbb{T}^d)$ (by multiplication) and the Dirac operator $D$ which acts on the dense subspace of smooth sections of the spinor bundle.
	We identify $S(\mathbb{T}^d)$ with the trivial bundle
	$\mathbb{T}^d \otimes V$, where $V := \mathbb{C}^{\lfloor \nicefrac{d}{2} \rfloor}$, and we write $\mathcal{H} := \mathrm{L}^2(\mathbb{T}^d) \otimes V \cong \mathrm{L}^2(S(\mathbb{T}^d))$ for the Hilbert space.
	Recall that $D = -i \sum_{\mu = 1}^d \partial_\mu \otimes \gamma^\mu$ and that the spectrum of $D$ (which is point-spectrum only) is given by $\sigma(D) = \{\pm (n_1^2 + \cdots + n_d^2)^{\nicefrac{1}{2}} \, : \, n_i \in \mathbb{Z}\}$. 
	The Dirac operator gives rise to the distance function \eqref{eqn:DistanceFormula} on the state space $\mathcal{S}(\mathrm{C}(\mathbb{T}^d))$ which metrizes the weak$^*$-topology on it and recovers the usual Riemannian distance on $\mathbb{T}^d$ when restricted to pure states.
	
	For any $\Lambda \geq 0$, let $P_\Lambda$ be the orthogonal projection to the subspace of $\mathcal{H}$ spanned by the eigenspinors $e_\lambda$ of the eigenvalues $\lambda$ with $|\lambda| \leq \Lambda$.
	%We call $P_\Lambda$ a \emph{spectral projection}.
	More concretely, we have $P_\Lambda \mathcal{H} = \mathrm{span} \{e_n \, : \, n \in \mathbb{Z}^d, \|n\| \leq \Lambda\} \otimes V$, with $e_n(x) := e^{i n \cdot x}$, for all $x \in \mathbb{T}^d$.
	The spectral projection $P_\Lambda$ gives rise to the following \emph{operator system spectral triple}:
	\begin{align}%\label{eqn:TruncatedSpectralTripleTorus}
		\left( P_\Lambda \mathrm{C}^\infty(\mathbb{T}^d) P_\Lambda, P_\Lambda \mathcal{H}, P_\Lambda D P_\Lambda \right)
	\end{align}
	We use the notation $\mathrm{C}(\mathbb{T}^d)^{(\Lambda)} := P_\Lambda \mathrm{C}^\infty(\mathbb{T}^d) P_\Lambda$ and write $D_\Lambda := P_\Lambda D P_\Lambda$. We also abbreviate $d_\Lambda := d_{P_\Lambda \mathrm{C}^\infty(\mathbb{T}^d) P_\Lambda}$ for the distance function defined in (\ref{eqn:DistanceFormula-trunc}).
	
	Observe that elements $T$ of the operator system $\mathrm{C}(\mathbb{T}^d)^{(\Lambda)}$ are of the form $T = (t_{k-l})_{k, l \in \overbar{\mathrm{B}}_\Lambda^\mathbb{Z}}$, where $\overbar{\mathrm{B}}_\Lambda^\mathbb{Z} := \overbar{\mathrm{B}}_\Lambda \cap \mathbb{Z}^d$ is the set of $\mathbb{Z}^d$-lattice points in the closed ball of radius $\Lambda$, and where $t_{k-l} = \langle e_k, T e_l \rangle$.
	In particular, $\mathrm{C}(\mathbb{T}^1)^{(\Lambda)}$ is the operator system of $(2\lfloor\Lambda\rfloor +1) \times (2\lfloor\Lambda\rfloor +1)$-Toeplitz matrices which was investigated at length in \cite{CvS21} and \cite{Far21}.
	
	The candidate for the map $\rho_\Lambda : \mathrm{C}^\infty(\mathbb{T}^d) \rightarrow \mathrm{C}(\mathbb{T}^d)^{(\Lambda)}$ in \autoref{defn:C1-approx} is canonically inherent in the problem, namely, it is the compression map given by $\rho_\Lambda(f) = P_\Lambda f P_\Lambda$.
	It is easy to see that this map is positive, unital and $\mathrm{C}^1$-contractive (\autoref{lem:PropertiesOfTheCompressionMap}).
	It is, however, less obvious what the candidate for the map $\sigma_\Lambda : \mathrm{C}(\mathbb{T}^d)^{(\Lambda)} \rightarrow \mathrm{C}^\infty(\mathbb{T}^d)$ should be.
	Inspired by the choice of the map given in \cite{vS21} in the case of the circle, we propose the following map:
	\begin{align}
		\sigma_\Lambda(T) := \frac{1}{\mathcal{N}_\mathrm{B}(\Lambda)} \mathrm{Tr} \left( \ket{\psi} \bra{\psi} \alpha(T) \right)
	\end{align}
	Here, $\alpha$ is the $\mathbb{T}^d$-action (\ref{eqn:TorusActionOnOperatorSystem}) on $\mathrm{C}(\mathbb{T}^d)^{(\Lambda)}$, the vector $\ket{\psi}$ is given by $\ket{\psi} = \sum_{n \in \overbar{\mathrm{B}}_\Lambda^\mathbb{Z}} e_n$, and $\mathcal{N}_\mathrm{B}(\Lambda) := \# \overbar{\mathrm{B}}_\Lambda \cap \mathbb{Z}^d  = \# \overline{\mathrm{B}}_\Lambda^\mathbb{Z}$ is the number of $\mathbb{Z}^d$-lattice points in the closed ball of radius $\Lambda$.
	Note that the image of an element $T = (t_{k-l})_{k,l \in \overbar{\mathrm{B}}_\Lambda^\mathbb{Z}}$ of the operator system $\mathrm{C}(\mathbb{T}^d)^{(\Lambda)}$ under this map is a function on $\mathbb{T}^d$ as follows:
	\begin{align}
		\sigma_\Lambda(T)(x) 
		&= \frac{1}{\mathcal{N}_\mathrm{B}(\Lambda)} \mathrm{Tr} \left( (1)_{k,l \in \overbar{\mathrm{B}}_\Lambda^\mathbb{Z}} (t_{k-l} e^{i(k-l) \cdot x})_{k,l \in \overbar{\mathrm{B}}_\Lambda^\mathbb{Z}} \right) \\
		&= \frac{1}{\mathcal{N}_\mathrm{B}(\Lambda)} \mathrm{Tr} \left(\left( \sum_{m \in \overbar{\mathrm{B}}_\Lambda^\mathbb{Z}} t_{m-l} e^{i(m-l) \cdot x} \right)_{k,l \in \overbar{\mathrm{B}}_\Lambda^\mathbb{Z}} \right) \\
		&= \frac{1}{\mathcal{N}_\mathrm{B}(\Lambda)} \sum_{m,n \in \overbar{\mathrm{B}}_\Lambda^\mathbb{Z}} t_{m-n} e^{i(m-n) \cdot x},
	\end{align}
	for all $x \in \mathbb{T}^d$.
	One may consider \cite[Section 2]{Rie04} as another instance of inspiration for this choice of map $\sigma_\Lambda$ by realizing that the map $\sigma_\Lambda$ is the formal adjoint of the map $\rho_\Lambda$ when the $\ast$-algebra $\mathrm{C}^\infty(\mathbb{T}^d)$ is equipped with the $\mathrm{L}^2$-inner product and the operator system $\mathrm{C}(\mathbb{T}^d)^{(\Lambda)}$ is equipped with the Hilbert--Schmidt inner product.
	Similarly as for $\rho_\Lambda$, it is easy to see that the map $\sigma_\Lambda$ is positive, unital and $\mathrm{C}^1$-contractive (\autoref{lem:PropertiesOfTheUpwardsMap}).
	
	In order to show that our choice of maps $\rho_\Lambda$ and $\sigma_\Lambda$ gives rise to a $\mathrm{C}^1$-approximate order isomorphism it remains to show that their compositions approximate the respective identities on $\mathrm{C}^\infty(\mathbb{T}^d)$ and $\mathrm{C}(\mathbb{T}^d)^{(\Lambda)}$ in Lipschitz seminorm.
	We show by direct computations (\autoref{lem:ComputationOfCompositions}) that the maps $\sigma_\Lambda \circ \rho_\Lambda$ and $\rho_\Lambda \circ \sigma_\Lambda$ act on $\mathrm{C}^\infty(\mathbb{T}^d)$, respectively $\mathrm{C}(\mathbb{T}^d)^{(\Lambda)}$ as follows:
	\begin{align}
		\sigma_\Lambda \circ \rho_\Lambda (f) &= \left(\mathfrak{m}_\Lambda \widehat{f}\right)\widecheck{\vphantom{f}} =: \mathcal{F}_{\mathfrak{m}_\Lambda} (f) \\
		\rho_\Lambda \circ \sigma_\Lambda (T) &= \left(\mathfrak{m}_\Lambda(k-l) t_{k-l}\right)_{k, l \in \overbar{\mathrm{B}}_\Lambda^\mathbb{Z}} =: \mathcal{S}_{\mathfrak{m}_\Lambda} (T)
	\end{align}
	The map $\mathcal{F}_{\mathfrak{m}_\Lambda}$ is known as \emph{Fourier multiplication} and the map $\mathcal{S}_{\mathfrak{m}_\Lambda}$ as \emph{Schur multiplication}, respectively with \emph{symbol} 
	\begin{align}\label{eqn:Symbol-m}
		\mathfrak{m}_\Lambda(n) := \frac{\mathcal{N}_\mathrm{L}(\Lambda,n)}{\mathcal{N}_\mathrm{B}(\Lambda)},
	\end{align}
	where $\mathcal{N}_\mathrm{L}(\Lambda,n) := \# \overbar{\mathrm{B}}_\Lambda \cap \overbar{\mathrm{B}}_\Lambda(n) \cap \mathbb{Z}^d = \# \overline{\mathrm{L}}_\Lambda^\mathbb{Z}(n)$ is the number of $\mathbb{Z}^d$-lattice points in the intersection of the closed ball of radius $\Lambda$ with a copy of itself translated by $n$ (we call this intersection a \emph{lense} and denote the set of $\mathbb{Z}^d$-lattice points in it by $\overline{\mathrm{L}}_\Lambda^\mathbb{Z}(n)$).
	
	We denote the compression of the Dirac operator by $D_\Lambda := P_\Lambda D P_\Lambda$.
	We apply an ``antiderivative trick'' (\autoref{lem:CompositionsInTermsOfCommutators}) to see that for obtaining estimates of the maps $\mathrm{id}_{\mathrm{C}^\infty(\mathbb{T}^d)} - \mathcal{F}_{\mathfrak{m}_\Lambda}$ and $\mathrm{id}_{\mathrm{C}(\mathbb{T}^d)^{(\Lambda)}} - \mathcal{S}_{\mathfrak{m}_\Lambda}$ in Lipschitz seminorm, one needs to estimate the following two maps:
	\begin{align}
		%	\label{eq:Fw}
		\F_{\mathfrak{w}_\Lambda} := \frac i2 \sum_{\mu=1}^d \F_{\mathfrak{w}^\mu_\Lambda} \otimes \{ \gamma^\mu,\cdot \}& : [D, \mathrm{C}^\infty(\T^d)] \to \mathrm{C}^\infty(\T^d);\\
		%	\label{eq:Sw}
		\S_{\mathfrak{w}_\Lambda} := \frac i2 \sum_{\mu=1}^d \S_{\mathfrak{w}^\mu_\Lambda} \otimes \{ \gamma^\mu,\cdot \}&: [D, \mathrm{C}(\T^d)^{(\Lambda)}] \to \mathrm{C}(\T^d)^{(\Lambda)},  
	\end{align}
	where $\mathcal{F}_{\mathfrak{w}_\Lambda^\mu}$ and $\mathcal{S}_{\mathfrak{w}_\Lambda^\mu}$ are now respectively Fourier and Schur multiplication with the symbol
	\begin{align}\label{eqn:Symbol_w}
		\mathfrak{w}_\Lambda^\mu (n) = 
		\begin{cases}
			0, \text{ if } n = 0 \\
			(1-\mathfrak{m}_\Lambda(n)) \frac{n_\mu}{\|n\|^2}, \text{ if } n \neq 0.
		\end{cases}
	\end{align}
	A variation of the classical Bo\.zejko-Fendler \emph{transference theorem} for Fourier and Schur multipliers (\autoref{lem:Sw-Fw-norms}) then shows that the cb-norm of $\S_{\mathfrak{w}_\Lambda}$ is bounded by the cb-norm of $\F_{\mathfrak{w}_\Lambda}$.
	Since the latter map takes values in a commutative $\mathrm{C}^*$-algebra its cb-norm coincides with its norm, so this is what is left to estimate.
	
	It is not hard to see that $\mathcal{F}_{\mathfrak{m}_\Lambda}$ is an approximate identity controlled in Lipschitz seminorm, if the convolution kernel $K_{\mathfrak{m}_\Lambda} = \widecheck{\mathfrak{m}}_\Lambda$ is a \emph{good kernel} (\autoref{lem:GoodKernelsGiveDiracEstimate}). 
	We show that this is indeed the case by exploiting the fact that $K_{\mathfrak{m}_\Lambda}$ is the square of the \emph{spherical Dirichlet kernel}, whence positive.
	Since this is analogous to the 1-dimensional case, we call $K_{\mathfrak{m}_\Lambda}$ the \emph{spectral Fejér kernel}.
	As outlined above, our main result that $(\rho_\Lambda, \sigma_\Lambda)$ is a $\mathrm{C}^1$-approximate order isomorphism and hence that spectral truncations of the $d$-torus converge for all $d \geq 1$ is now simply a corollary of the fact that the \emph{spectral Fejér kernel} is \emph{good}.
	
	In the last section, we give a computation of the propagation number of the operator system $\mathrm{C}(\mathbb{T}^d)^{(\Lambda)}$.
	We conclude by pointing out some obstacles on the road to determining its operator system dual.
	In particular, we argue that, for $d \geq 2$, the operator system dual cannot quite be the one that we would have expected.
	
	\bigskip
	
	\subsection*{Acknowledgements}
	We thank Nigel Higson for the comments at an early stage of this project.
	ML thanks Gerrit Vos for an introduction to Fourier and Schur multipliers and Dimitris Gerontogiannis for fruitful discussions.
	After having published a first preprint version of this article we were pointed to \autoref{prop:SpectralFejerKernel}, which makes many technicalities obsolete and at the same time enables us to prove convergence of spectral truncations of tori in any dimension (also higher than $3$). 
	We are indepted to Yvann Gaudillot-Estrada for this contribution. 
	Furthermore, we thank Jens Kaad and Tirthankar Bhattacharyya for their input.
	This work was funded by NWO under grant OCENW.KLEIN.376.

	\section{Preliminaries}
	
	\subsection{Actions and commutators}
	
	We spell out a few simple facts used throughout this article.
	Recall the usual action of $\mathrm{C}(\mathbb{T}^d)$ on $\mathcal{H}$:
	\begin{align}
		f (g \otimes v) = (fg) \otimes v = \sum_{n \in \mathbb{Z}^d}  \sum_{m \in \mathbb{Z}^d} \widehat{f}(n-m) \widehat{g}(m) e_n \otimes v
	\end{align}
	This induces an action of $\mathrm{C}(\mathbb{T}^d)^{(\Lambda)}$ on $P_\Lambda \mathcal{H}$:
	\begin{align}
		T \left( \sum_{k \in \overbar{\mathrm{B}}_\Lambda^\mathbb{Z}} a_k e_k \otimes v \right) = \sum_{k \in \overbar{\mathrm{B}}_\Lambda^\mathbb{Z}} \sum_{l \in \overbar{\mathrm{B}}_\Lambda^\mathbb{Z}} t_{k-l} a_l e_k \otimes v
	\end{align}
	Furthermore, we have the standard $\mathbb{T}^d$-action on $\mathrm{C}(\mathbb{T}^d)$ (as a subalgebra of $\mathcal{B}(\mathcal{H})$):
	\begin{align}
		\alpha_\theta(f) = \sum_{n \in \mathbb{Z}^d} \widehat{f}(n) e_{n} e^{i n \cdot \theta}
	\end{align}
	This induces an action of $\mathbb{T}^d$ on $\mathrm{C}(\mathbb{T}^d)^{(\Lambda)}$ (as an operator subsystem of $\mathcal{B}(P_\Lambda \mathcal{H})$):
	\begin{align}\label{eqn:TorusActionOnOperatorSystem}
		\alpha_\theta(T) = \left( t_{k-l} e^{i (k-l) \cdot \theta}\right)_{k,l \in \overbar{\mathrm{B}}_\Lambda^\mathbb{Z}}
	\end{align}
	
	Recall that the following holds:
	\begin{align}
		[D,f]
		= \sum_{\mu = 1}^d \sum_{n \in \mathbb{Z}^d} n_\mu \widehat{f}(n) e_n \otimes \gamma^\mu
	\end{align}
	Similarly, we have:
	\begin{align}\label{eqn:CompressedDiracCommutator}
		[D_\Lambda, T]
		= \sum_{\mu = 1}^d \left( (k_\mu - l_\mu) t_{k-l} \right)_{k,l \in \overbar{\mathrm{B}}_\Lambda^\mathbb{Z}} \otimes \gamma^\mu
	\end{align}

	\subsection{Good kernels}
	
	We follow the convention in \cite{SS03} and call an approximate identity in the Banach $\ast$-algebra $\mathrm{L}^1(\mathbb{T}^d)$ (with convolution) a \emph{good kernel}:
	
	\begin{definition}
		For all $\Lambda > 0$, let $K_\Lambda \in \mathrm{L}^1(\mathbb{T}^d)$.
		The family $\{K_\Lambda\}_{\Lambda > 0}$ is called a \emph{good kernel}, if the following holds:
		\begin{enumerate}
			\item $\int_{\mathbb{T}^d} K_\Lambda(x) dx = 1$ and
			\item for all $\delta > 0$, we have that $\int_{\mathbb{T}^d \setminus \mathrm{B}_\delta(0)} K_\Lambda(x) dx \rightarrow 0$, as $\Lambda \rightarrow \infty$.
		\end{enumerate}
		
	\end{definition}
	
	Good kernels provide a way to approximate the identity on $\mathrm{C}^\infty(\mathbb{T}^d)$ not only in $\mathrm{L}^1$- and $\sup$-norm, but also in Lipschitz seminorm:
	
	\begin{lemma}\label{lem:GoodKernelsGiveDiracEstimate}
		If $K_\Lambda$ is a \emph{good kernel}, then, for all $f \in \mathrm{C}^\infty(\mathbb{T}^d)$, the following holds:
		\begin{align}
			\|f - K_\Lambda \ast f\| \leq \gamma_\Lambda \|[D,f]\|,
		\end{align}
		where $\gamma_\Lambda \rightarrow 0$ as $\Lambda \rightarrow \infty$.
	\end{lemma}
	
	\begin{proof}
		The proof is as in \cite[Lemma 5.13]{Ber19}.
		For all $x \in \mathbb{T}^d$, we have:
		\begin{align}
			|K_\Lambda \ast f(x) - f(x)|
			&\leq \int_{\mathbb{T}^d} | K_\Lambda(y) (f(x-y)-f(x)) | dy \\
			&\leq \int_{\mathbb{T}^d} |K_\Lambda(y)| \|y\| \|f\|_{\text{Lip}} dy \\
			&= \int_{\mathbb{T}^d} |K_\Lambda(y)| \|y\| dy \|[D,f]\|
		\end{align}
		We set $\gamma_\Lambda := \int_{\mathbb{T}^d} |K_\Lambda(y)| \|y\| dy$.
		Let $\varepsilon > 0$.
		Let $\Lambda_0$ be large enough such that, for all $\Lambda \geq \Lambda_0$, $\int_{\|y\| \geq \varepsilon} |K_\Lambda(y)| dy < \varepsilon$.
		Then, for $\Lambda \geq \Lambda_0$, we obtain:
		\begin{align}
			\gamma_\Lambda
			&= \int_{\|y\| \geq \varepsilon} |K_\Lambda(y)| \|y\| dy + \int_{\|y\| < \varepsilon} |K_\Lambda(y)| \|y\| dy \\
			&\leq d \left( \int_{\|y\| \geq \varepsilon} |K_\Lambda(y)| dy + \varepsilon \int_{\|y\| < \varepsilon} |K_\Lambda(y)| dy \right) \\
			&\leq d \left( \varepsilon + \varepsilon C \right),
		\end{align}
		where $C = \sup_{\Lambda > 0} \int_{\mathbb{T}^d} |K_\Lambda(y)| dy < \infty$.
		%	The first estimate is trivial and the second estimate follows from \ref{AI3} and \ref{AI1}.
	\end{proof}

	\subsection{Fourier and Schur multipliers}
	
	Let $\Gamma$ be a discrete group and let $\lambda : \Gamma \rightarrow \mathcal{B}(\ell^2(\Gamma))$ be its left-regular representation given by $\lambda_g f (h) = f(gh)$.
	We denote by $\mathrm{C}_\lambda^\ast(\Gamma)$ the reduced group $\mathrm{C}^\ast$-algebra, i.e.\@ the completion of the group ring $\mathbb{C}[\Gamma]$ in $\mathcal{B}(\ell^2(\Gamma))$ with respect to the norm $\|x\|_\mathrm{red} := \|\lambda(x)\|_{\mathcal{B}(\ell^2(\Gamma))}$.
	A function $\varphi : \Gamma \rightarrow \mathbb{C}$ gives rise to a \emph{multiplier} on the group ring as follows:
	\begin{align}
		\mathbb{C}[\Gamma] &\rightarrow \mathbb{C}[\Gamma] \\
		\sum_{g \in \Gamma} a_g g &\mapsto \sum_{g \in \Gamma} \varphi(g) a_g g
	\end{align}
	If this map extends to a bounded linear map $\mathcal{M}_\varphi : \mathrm{C}_\lambda^\ast(\Gamma) \rightarrow \mathrm{C}_\lambda^\ast(\Gamma)$ we call this extension the \emph{multiplier on $\mathrm{C}_\lambda^\ast(\Gamma)$ with symbol $\varphi$}.
	We record the obvious fact that if $\varphi$ is finitely supported it always induces a multiplier on $\mathrm{C}_\lambda^\ast(\Gamma)$.
	
	Recall that if $\Gamma$ is abelian, then $\mathrm{C}_\lambda^\ast(\Gamma) = \mathrm{C}^\ast(\Gamma) \cong \mathrm{C}(\widehat{\Gamma})$, where $\widehat{\Gamma}$ is the Pontryagin dual of $\Gamma$.
	In this case, we call the multiplier on $\mathrm{C}(\widehat{\Gamma})$ the \emph{Fourier multiplier with symbol $\varphi$} and denote it by $\mathcal{F}_\varphi$.
	The Fourier multiplier takes on the following form:
	\begin{align}\label{eqn:FourierMultiplierOnAbelianGroup}
		\mathcal{F}_\varphi(f) 
		= \left(g \mapsto \varphi(g) \widehat{f}(g)\right)\widecheck{\vphantom{f}}
		= \widecheck{\varphi} \ast f
	\end{align}
	See e.g.\@ \cite[Chapter 6]{HR94} for the relevant Fourier theory of locally compact abelian groups.
	
	Let $k : \Gamma \times \Gamma \rightarrow \mathbb{C}$ be a function, also called a \emph{kernel}.
	A kernel $k$ gives rise to a linear map with domain $\mathcal{B}(\ell^2(\Gamma))$ given by $\mathcal{S}_k : (t_{g,h})_{g,h \in \Gamma} \mapsto (k(g,h)t_{g,h})_{g,h \in \Gamma}$, where $t_{g,h} = \langle  \delta_g, T \delta_h \rangle$, for $T \in \mathcal{B}(\ell^2(\Gamma))$.
	If this map is bounded with range in $\mathcal{B}(\ell^2(\Gamma))$, we call it a \emph{Schur multiplier}.
	See e.g.\@ \cite{TT10} for a survey and \cite[Chapter 5]{Pis95} as a standard reference which includes a discussion of the connection with Grothendieck's theorem.
	We collect some well-known facts about Schur multipliers.
	
	\begin{proposition}\label{prop:CharacterizationsSchurMultiplier}
		Let $k : \Gamma \times \Gamma \rightarrow \mathbb{C}$ be a kernel.
		Then the following are equivalent:
		\begin{enumerate}[label=(\roman*)]
			\item $\mathcal{S}_k$ is a Schur multiplier of norm $\|\mathcal{S}_k\| \leq 1$.
			\item $\mathcal{S}_k$ is a completely bounded Schur multiplier of $\mathrm{cb}$-norm $\|\mathcal{S}_k\|_\mathrm{cb} \leq 1$.
			\item There exists a Hilbert space $\mathcal{H}$ and families of vectors $\{\xi_g\}_{g \in \Gamma}, \{\eta_h\}_{h \in \Gamma} \subset \mathcal{H}$ with $\|\xi_g\|, \|\eta_h\| \leq 1$ such that $k(g,h) = \langle \xi_g, \eta_h \rangle$, for all $g, h \in \Gamma$.
		\end{enumerate}
	\end{proposition}
	
	For an elementary proof of the equivalence of (i) and (ii), we refer to \cite[Theorem 8.7 and Corollary 8.8]{Pau02}.
	A proof of the equivalence of (ii) and (iii) can be found e.g.\@ in \cite[Theorem D.4]{BO08}, which we now sketch:
	Assuming that $\|\mathcal{S}_k\|_{\mathrm{cb}} \leq 1$, Wittstock's factorization theorem gives a factorization of $\mathcal{S}_k$ through $\mathcal{B}(\mathcal{H})$, for some Hilbert space $\mathcal{H}$, which allows to construct appropriate $\xi_g$ and $\eta_h$.
	For the converse implication, the map $\mathcal{S}_k$ is factorized through $\mathcal{B}(\ell^2(\Gamma) \otimes \mathcal{H})$ as $\mathcal{S}_k(T) = V^\ast (T \otimes \mathbf{1}_\mathcal{H}) W$, for the contractions $V \delta_g := \delta_g \otimes \xi_g$ and $W \delta_h := \delta_h \otimes \eta_h$.
	The same works when tensoring with $\mathbf{1}_{M_n}$, for arbitrary $n \in \mathbb{N}$, which shows complete contractivity of $\mathcal{S}_k$.
	
	We are mainly interested in Schur multipliers $\mathcal{S}_k$ induced by a function $\varphi : \Gamma \rightarrow \mathbb{C}$, i.e.\@ $k(g,h) := \varphi(gh^{-1})$.
	We call such a Schur multiplier a \emph{Schur multiplier with symbol $\varphi$} and slightly abuse notation to denote it by $\mathcal{S}_\varphi$.
	It is easy to see that $\mathcal{S}_\varphi\big|_{\mathrm{C}_\lambda^\ast(\Gamma)} = \mathcal{M}_\varphi$.
	Indeed, let $f \in \mathrm{C}_\lambda^\ast(\Gamma)$ and $\{\delta_g\}_{g \in \Gamma}$ be an orthonormal basis for $\ell^2(\Gamma)$.
	Then the matrix associated to $f$ (viewed as an element of $\mathcal{B}(\ell^2(\Gamma))$) is a Toeplitz matrix in the following sense:
	\begin{align}
		\langle \delta_g, f \delta_h \rangle
		= \langle \delta_g , \sum_{\gamma \in \Gamma} f_\gamma \lambda_\gamma(\delta_h) \rangle 
		=  \sum_{\gamma \in \Gamma} \langle \delta_g , f_\gamma \delta_{\gamma h} \rangle 
		= f_{gh^{-1}}
	\end{align}
	It follows that the matrix associated to $\mathcal{M}_\varphi(f)$ is the following Toeplitz matrix:
	\begin{align}
		\langle \delta_g, \mathcal{M}_\varphi(f) \delta_h \rangle
		= \varphi(gh^{-1})f_{gh^{-1}},
	\end{align}
	which shows that $\mathcal{S}_\varphi((f_{gh^{-1}})_{g,h \in \Gamma}) = ((\mathcal{M}_\varphi(f))_{g,h})_{g,h \in \Gamma}$.

	\section{Convergence of Spectral Truncations of the $d$-Torus}
	
	The goal of this section is to prove that the maps $\rho_\Lambda : \mathrm{C}^\infty(\mathbb{T}^d) \rightarrow \mathrm{C}(\mathbb{T}^d)^{(\Lambda)}$, given by the compression $\rho_\Lambda(f) := P_\Lambda f P_\Lambda$, and $\sigma_\Lambda : \mathrm{C}(\mathbb{T}^d)^{(\Lambda)} \rightarrow \mathrm{C}^\infty(\mathbb{T}^d)$, given by $\sigma_\Lambda (T) := \frac{1}{\mathcal{N}_\mathrm{B}(\Lambda)} \mathrm{Tr}(\ket{\psi}\bra{\psi}\alpha(T))$, form a $\mathrm{C}^1$-approximate order isomorphism in the sense of Definition \ref{defn:C1-approx}.

	\subsection{A candidate for the $\mathrm{C}^1$-approximate order isomorphism}
	
	We begin by checking unitality, positivity and $\mathrm{C}^1$-contractivity for the maps $\rho_\Lambda$ and $\sigma_\Lambda$.
	
	\begin{lemma}\label{lem:PropertiesOfTheCompressionMap}
		The map $\rho_\Lambda : \mathrm{C}^\infty(\mathbb{T}^d) \rightarrow \mathrm{C}(\mathbb{T}^d)^{(\Lambda)}$ is unital, positive and $\mathrm{C}^1$-contractive.
	\end{lemma}
	
	\begin{proof}
		Unitality is clear as 
		$P_\Lambda \mathbf{1} P_\Lambda \left(\sum_{n \in \overline{\mathrm{B}}_\Lambda^\mathbb{Z}} a_{n} e_{n} \otimes s_n\right) 
		= \sum_{n \in \overline{\mathrm{B}}_\Lambda^\mathbb{Z}} a_{n} e_{n} \otimes s_n$, 
		for all $\sum_{n \in \overline{\mathrm{B}}_\Lambda^\mathbb{Z}} a_{n} e_{n} \otimes s_n \in P_\Lambda \mathcal{H}$.
		Also positivity is obvious since $\langle PaP\varphi, P\varphi \rangle_{P\mathcal{K}} = \langle a P\varphi, P\varphi \rangle_{\mathcal{K}} \geq 0$, for any Hilbert space $\mathcal{K}$, any projection $P \in \mathcal{B}(\mathcal{K})$ and any positive operator $a \in \mathcal{B}(\mathcal{K})_+$.	
		Contractivity in norm follows from Plancherel's theorem:
		\[
		\|P_\Lambda f P_\Lambda\|^2 
		= \sup_{\underset{\|\varphi\| \leq 1}{\varphi \in P_\Lambda \mathcal{H}}} \|P_\Lambda f P_\Lambda \varphi\|^2 
		\leq \displaystyle{\sum_{n \in \overline{\mathrm{B}}_\Lambda^\mathbb{Z}}} \left|\widehat{f}(n)\right|^2 \leq  \sum_{n \in \mathbb{Z}^d} \left|\widehat{f}(n)\right|^2 
		= \|f\|^2
		\]
		For contractivity in Lipschitz seminorm, we first observe that $\rho_\Lambda$ commutes with $[D,\cdot]$ in the following sense, which is an immediate consequence of the fact that $D$ commutes with $P_\Lambda$:
		\begin{align}
			[D_\Lambda, \rho_\Lambda(f)]
			= P_\Lambda [D,f] P_\Lambda
		\end{align}
		This immediately gives $\|[D,\cdot]\|$-contractivity:
		\begin{align}
			\|[D_\Lambda,\rho_\Lambda(f)]\|
			= \| P_\Lambda [D,f] P_\Lambda \|
			\leq \| [D,f] \|
		\end{align}
		\black
	\end{proof}
	
	\begin{lemma}\label{lem:PropertiesOfTheUpwardsMap}
		The map $\sigma_\Lambda : \mathrm{C}(\mathbb{T}^d)^{(\Lambda)} \rightarrow \mathrm{C}^\infty(\mathbb{T}^d)$ is unital, positive and $\mathrm{C}^1$-contractive.
	\end{lemma}
	
	\begin{proof}
		Unitality is clear as, for all $x \in \mathbb{T}^d$, we have:
		\begin{align}
			\sigma_\Lambda(\mathbf{1})(x)
			= \frac{1}{\mathcal{N}_{\mathrm{B}}(\Lambda)} \mathrm{Tr} \left( \ket{\psi} \bra{\psi} \alpha_x(\mathbf{1}) \right)
			= \frac{1}{\mathcal{N}_{\mathrm{B}}(\Lambda)} \mathrm{Tr} \left( \ket{\psi} \bra{\psi} \mathbf{1} \right)
			= 1
		\end{align}
		Positivity is also immediate from the definition.
		Namely, let $T \in \mathrm{C}(\mathbb{T}^d)^{(\Lambda)}_+$ be a positive operator on $P_\Lambda \mathcal{H}$ and let $P_\Lambda \mathcal{H} \ni \zeta \mapsto Q_T(\zeta) := \langle \zeta, T \zeta \rangle$ be its associated quadratic form.
		For $\zeta = \sum_{n \in \overbar{\mathrm{B}}_\Lambda^\mathbb{Z}} e_n$, we obtain:
		\begin{align}
			0
			\leq \frac{1}{\mathcal{N}_{\mathrm{B}}(\Lambda)} Q_T(\zeta)(x)
			= \frac{1}{\mathcal{N}_{\mathrm{B}}(\Lambda)} \mathrm{Tr} \left( \ket{\psi} \bra{\psi} \left(e_{-(m-n)} t_{m-n}\right)_{m,n} \right)
			= \sigma_\Lambda(T)(x),
		\end{align}
		for all $x \in \mathbb{T}^d$.
		For contractivity, we compute:
		\begin{align}
			| \sigma_\Lambda (T)(x) | &\leq \frac{1}{\mathcal{N}_{\mathrm{B}}(\Lambda)} \Tr  (\ket{\psi} \bra{\psi}) \|  \alpha_x (T) \| = \| T \|  %\\
		\end{align}
		For contractivity in Lipschitz seminorm, we first observe that $\sigma_\Lambda$ commutes with $[D,\cdot]$ in the following sense, which is an easy consequence of (\ref{eqn:CompressedDiracCommutator}):
		\begin{align}
			[D, \sigma_\Lambda(T)] 
			&= -i \sum_{\mu = 1}^d \frac{1}{\mathcal{N}_\mathrm{B}(\Lambda)} \sum_{n \in \overbar{\mathrm{B}}_\Lambda^\mathbb{Z}} i n_\mu t_{n} e_{n} \otimes \gamma^\mu \\
			&= \sum_{\mu = 1}^d \frac{1}{\mathcal{N}_\mathrm{B}(\Lambda)} \mathrm{Tr} \left( \ket{\psi}\bra{\psi} \alpha \left( ((n_\mu-m_\mu) t_{n-m})_{n,m \in \overbar{\mathrm{B}}_\Lambda^\mathbb{Z}} \right) \right) \otimes \gamma^\mu \\
			&= \sigma_\Lambda \otimes \mathbf{1} ([D_\Lambda, T])
		\end{align}
		This immediately gives $\|[D,\cdot]\|$-contractivity:
		\begin{align}
			\| [D, \sigma_\Lambda(T)] \| 
			= \| \sigma_\Lambda \otimes \mathbf{1} \left( [D_\Lambda, T] \right) \|
			\leq \| [D_\Lambda, T] \|
		\end{align}
		\black
	\end{proof}
	
	We now compute the compositions $\sigma_\Lambda \circ \rho_\Lambda$ and $\rho_\Lambda \circ \sigma_\Lambda$:
	
	\begin{lemma}\label{lem:ComputationOfCompositions}
		The two compositions  $\sigma_\Lambda \circ \rho_\Lambda : \mathrm{C}^\infty(\mathbb{T}^d) \rightarrow \mathrm{C}^\infty(\mathbb{T}^d)$ and $\rho_\Lambda \circ \sigma_\Lambda : \mathrm{C}(\mathbb{T}^d)^{(\Lambda)} \rightarrow \mathrm{C}(\mathbb{T}^d)^{(\Lambda)}$ are given respectively by the Fourier multiplier and by the Schur multiplier with the symbol $\mathfrak{m}_\Lambda$:
		\begin{align}
			\sigma_\Lambda \circ \rho_\Lambda &= \mathcal{F}_{\mathfrak{m}_\Lambda} \\
			\rho_\Lambda \circ \sigma_\Lambda &= \mathcal{S}_{\mathfrak{m}_\Lambda}
		\end{align}
	\end{lemma}
	
	\begin{proof}
		Both identities are just simple computations:
		\begin{align}
			\sigma_\Lambda \circ \rho_\Lambda (f) (x)
			&= \frac{1}{\mathcal{N}_{\mathrm{B}}(\Lambda)} \mathrm{Tr} \left( \ket{\psi} \bra{\psi} \alpha_x(P_\Lambda f P_\Lambda) \right) \\
			&= \frac{1}{\mathcal{N}_{\mathrm{B}}(\Lambda)} \mathrm{Tr} \left( \ket{\psi} \bra{\psi} \alpha_x \left( \left( \widehat{f}({k-l})\right)_{k,l \in \overline{\mathrm{B}}_\Lambda^\mathbb{Z}}\right) \right) \\
			&= \frac{1}{\mathcal{N}_{\mathrm{B}}(\Lambda)} \sum_{k, l \in \overline{\mathrm{B}}_\Lambda^\mathbb{Z}} \widehat{f}({k-l}) e^{i (k-l) \cdot x} \\
			&= \sum_{n \in \overline{\mathrm{B}}_\Lambda^\mathbb{Z} - \overline{\mathrm{B}}_\Lambda^\mathbb{Z}} \frac{\mathcal{N}_{\mathrm{L}}(\Lambda, n)}{\mathcal{N}_{\mathrm{B}}(\Lambda)} \widehat{f}(n) e_n(x) \\ 
			\\
			\rho_{\Lambda} \circ \sigma_\Lambda (T)
			&= \rho_{\Lambda} \left(\frac{1}{\mathcal{N}_{\mathrm{B}}(\Lambda)} \mathrm{Tr} \left(\ket{\psi} \bra{\psi} \alpha_\bullet(T)\right)\right) \\
			&=  \rho_{\Lambda} \left( \frac{1}{\mathcal{N}_{\mathrm{B}}(\Lambda)} \sum_{k,l \in \overline{\mathrm{B}}_\Lambda^{\mathbb{Z}}} t_{k-l} e_{k-l} \right) \\
			&=  \rho_{\Lambda} \left( \sum_{n \in \overline{\mathrm{B}}_{\Lambda}^{\mathbb{Z}} - \overline{\mathrm{B}}_{\Lambda}^{\mathbb{Z}}} \frac{\mathcal{N}_{\mathrm{L}}(\Lambda,n)}{\mathcal{N}_{\mathrm{B}}(\Lambda)} t_{n} e_n \right) \\
			&= \left( \frac{\mathcal{N}_{\mathrm{L}}(\Lambda,m-n)}{\mathcal{N}_{\mathrm{B}}(\Lambda)} t_{m-n} \right)_{m,n \in \overline{\mathrm{B}}_{\Lambda}^{\mathbb{Z}}}
		\end{align}
	\end{proof}

	\subsection{A transference result}
	
	In order to show that the pair $(\rho_\Lambda, \sigma_\Lambda)$ is a $\mathrm{C}^1$-approximate order isomorphism it remains to check that the compositions $\sigma_\Lambda \circ \nolinebreak \rho_\Lambda$ and $\rho_\Lambda \circ \sigma_\Lambda$ approximate the identity respectively on $\mathrm{C}(\mathbb{T}^d)$ and $\mathrm{C}(\mathbb{T}^d)^{(\Lambda)}$ in Lipschitz seminorm.
	
	Recall the definition of $\mathfrak{w}_\Lambda^\mu(n)$, for $\mu = 1,\dots,d$ and $n \in \mathbb{Z}^d$ from (\ref{eqn:Symbol_w}).
	
	\begin{lemma}\label{lem:CompositionsInTermsOfCommutators}
		For every $f \in \mathrm{C}^\infty(\mathbb{T}^d)$ we have the following equality of bounded operators on the Hilbert space $\mathrm{L}^ 2(\mathbb{T}^d) \otimes V$:
		\begin{align}
			\left(f - \sigma_\Lambda \circ \rho_\Lambda (f)\right) \otimes \mathbf{1}
			= \frac{i}{2} \left(\sum_{\mu=1}^d \mathcal{F}_{\mathfrak{w}_\Lambda^\mu} \otimes \{\gamma^\mu,\cdot\}\right) ([D,f]),
		\end{align}
		where $\mathcal{F}_{\mathfrak{w}_\Lambda^\mu}$ is the Fourier multiplier on the $\ast$-algebra $\mathrm{C}^\infty(\mathbb{T}^d)$ with symbol $\mathfrak{w}_\Lambda^\mu$.
		
		Similarly, for every $T \in \mathrm{C}(\mathbb{T}^d)^{(\Lambda)}$, we have the following equality of bounded operators on the Hilbert space $P_\Lambda \mathrm{L}^ 2(\mathbb T^d) \otimes V$:
		\begin{align}
			\left(T - \rho_\Lambda \circ \sigma_\Lambda (T) \right) \otimes \mathbf{1}
			= \frac{i}{2} \left( \sum_{\mu=1}^{d}\mathcal{S}_{\mathfrak{w}_\Lambda^\mu} \otimes \{\gamma^\mu,\cdot\}\right) ([D_\Lambda,T]),
		\end{align}
		where $\mathcal{S}_{\mathfrak{w}_\Lambda^\mu}$ denotes Schur (i.e.\@ entrywise) multiplication on the operator system $\mathrm{C}(\mathbb{T}^d)^{(\Lambda)}$ with symbol $\mathfrak{w}_\Lambda^\mu$.
	\end{lemma}
	
	\proof 
	For the first claim, we have:
	\begin{align}
		f - \sigma_\Lambda \circ \rho_\Lambda (f)
		&=f- \mathcal{F}_{\mathfrak{m}_\Lambda} (f) = \sum_{n \in \Z^d}  (1-\mathfrak{m}_\Lambda(n) )  \widehat{f}(n) e_n
		\\
		&=  \sum_{\mu=1}^{d} \sum_{n \in \Z^d}  (1-\mathfrak{m}_\Lambda(n) )  \frac{n_\mu}{\| n\|^2}   \cdot  n_\mu \widehat{f}(n) e_n \\
		& =  \sum_{\mu,\nu=1}^{d} \mathcal{F}_{\mathfrak{w}^\mu_\Lambda} ( -i \partial_\nu f) \cdot \delta^{\mu\nu}.
	\end{align}
	The first claimed result then follows by writing $2\delta^{\mu\nu} \cdot \mathbf{1} = \{\gamma^\mu,\gamma^\nu\}$.
	
	The computation for the second claim is largely analogous:
	\begin{align}
		T - \rho_\Lambda \circ \sigma_\Lambda (T)
		&= T- \mathcal{S}_{\mathfrak{m}_\Lambda}(T) \\
		&= \sum_{\mu=1}^d \left( (1-\mathfrak{m}_\Lambda(k-l)) \frac{k_\mu-l_\mu }{\|k-l\|^2} \cdot  (k_\mu-l_\mu)t_{k-l} \right)_{k,l \in \overbar{\mathrm{B}}_\Lambda^{\mathbb{Z}}} \\
		&= i \sum_{\mu,\nu=1}^d \mathcal{S}_{\mathfrak{w}^\mu_\Lambda}\left(  \left( (k_\nu-l_\nu) t_{k-l} \right )_{k,l\in \overbar{\mathrm{B}}_\Lambda^{\mathbb{Z}}} \right) \cdot \delta^{\mu\nu}.
	\end{align}
	The result now follows by combining the defining relations for the gamma-matrices as before with the expression (\ref{eqn:CompressedDiracCommutator}) for the operator $ [D_\Lambda,T] \in \mathcal{B}(P_\Lambda \mathcal{H})$.
	\endproof
	
	It is a classical result of Bo\.zejko and Fendler \cite{BF84} that for any discrete group $\Gamma$ and function $\varphi: \Gamma \to \C$ the cb-norm of Schur multiplication $S_\varphi$ on $\B(\ell^2(\Gamma))$ coincides with the cb-norm of Fourier multiplication $\F_\varphi$ on $\mathrm{C}^*_\lambda(\Gamma)$  (see also \cite[Theorem 6.4]{Pis95} and \cite[Proposition D.6]{BO08}). 
	However, the two linear maps obtained in \autoref{lem:CompositionsInTermsOfCommutators} act on the operator subsystems of differential forms on $\mathcal{H}$ and $P_\Lambda \mathcal{H}$, respectively, so the result of Bo\.zejko and Fendler does not apply directly. 
	We prove a variation on it which relates the cb-norms of the two linear maps which appear in \autoref{lem:CompositionsInTermsOfCommutators}:
	\begin{align}
		\label{eq:Fw}
		\F_{\mathfrak{w}_\Lambda} := \frac i2 \sum_{\mu=1}^d \F_{\mathfrak{w}^\mu_\Lambda} \otimes \{ \gamma^\mu,\cdot \}& : [D, \mathrm{C}^\infty(\T^d)] \to \mathrm{C}^\infty(\T^d);\\
		\label{eq:Sw}
		\S_{\mathfrak{w}_\Lambda} := \frac i2 \sum_{\mu=1}^d \S_{\mathfrak{w}^\mu_\Lambda} \otimes \{ \gamma^\mu,\cdot \}&: [D_\Lambda, \mathrm{C}(\T^d)^{(\Lambda)}] \to \mathrm{C}(\T^d)^{(\Lambda)},  
	\end{align}
	Here we consider $[D,\mathrm{C}^\infty(\T^d)]$ and $[D_\Lambda, \mathrm{C}(\T^d)^{(\Lambda)}]$ as (dense subsets of) operator systems in $\B(\mathrm{L}^2(\T^d) \otimes V)$. 
	
	\begin{lemma}
		\label{lem:Sw-Fw-norms}
		For the above two linear maps we have the following norm inequality:
		$$\| \S_{\mathfrak{w}_\Lambda} \| \leq \|  \F_{\mathfrak{w}_\Lambda}  \|
		$$
	\end{lemma}
	\proof
	We vary on the proof given in \cite[Theorem 6.4]{Pis95}. First identify $\mathrm{L}^2(\T^d) \otimes V \cong \ell^2(\Z^d) \otimes V$ using the Fourier basis $\{ e_n \}_{n \in \Z^d}$, and write $\H:=\ell^2(\Z^d) \otimes V$. Consider the unitary operator $U$ defined on $\H \otimes \H$ by a combination of a shift in Fourier space and a tensor flip in spinor space: 
	$$
	U (e_n \otimes v \otimes e_m \otimes v' ) = e_n \otimes v' \otimes e_{n+m} \otimes v
	$$
	Recall that an elementary matrix $E_{kl}$ ($k,l \in \Z^d$) acts on $\H$ as:
	\begin{align}
		E_{kl}(e_n \otimes v)& = \delta_{ln} e_k \otimes v,
		\intertext{in contrast to a generator $e_k$ in the group $\mathrm{C}^*$-algebra $\mathrm{C}^*(\Z^d)  \cong \mathrm{C}(\mathbb{T}^d)$, which acts as}
		e_k (e_n \otimes v) & = e_{n+k} \otimes v.
	\end{align}
	Note furthermore that under the identification $\mathrm{L}^ 2(\mathbb{T}^d) \cong \ell^ 2(\Z^ d)$ we have
	\begin{equation}
		\label{eq:dirac}
		D (e_n \otimes v ) = n e_n \otimes \gamma^ \mu v ,\qquad (n \in \Z^d, v \in V).
	\end{equation}
	We then find that
	\begin{align}
		U(E_{kl} \otimes \mathbf{1}_\H) U^* & = E_{kl} \otimes e_{k-l} \\
		U(E_{kl}  \gamma^\mu\otimes \mathbf{1}_\H) U^* & = E_{kl} \otimes e_{k-l} \gamma^\mu 
	\end{align}
	where $\gamma^\mu$ acts of course on the spinor space $V$. 
	Note that in view of Equation \eqref{eq:dirac} we also have $U([D,E_{kl}] \otimes \mathbf{1}_\H )U^* = E_{kl} \otimes [D,e_{k-l}] \in \mathcal{K}(\mathcal{H}) \otimes [D,\mathrm{C}^\infty(\mathbb{T}^d)]$. 
	
	Using this we may now show
	\begin{align}
		U \left( ( \S_{\mathfrak{w}_\Lambda}   \otimes \mathrm{id}_{\mathcal{B}(\mathcal{H})} )\big([D,E_{kl}] \otimes  \mathbf{1}_\H \big) \right)U^*
		&= \sum_\mu  U \left( \S_{\mathfrak{w}_\Lambda}  (\gamma^\mu (k-l)_\mu E_{kl})   \otimes  \mathbf{1}_\H   \right)U^*\\
		& = i \sum_\mu U \left( {\mathfrak{w}^\mu_\Lambda} (k-l) (k-l)_\mu E_{kl}   \otimes  \mathbf{1}_\H  \right)U^*\\
		&= i \sum_\mu  E_{kl} \otimes  {\mathfrak{w}^\mu_\Lambda} (k-l)  (k-l)_\mu e_{k-l} \\
		& =   E_{kl}  \otimes \F_{\mathfrak{w}_\Lambda}  ([D,e_{k-l}])\\
		&= (\mathrm{id}_{\mathcal{B}(\mathcal{H})} \otimes \F_{\mathfrak{w}_\Lambda}  )\left( U ([D, E_{kl}] \otimes \mathbf{1}_\H )U^* \right).
	\end{align}
	This extends by linearity to arbitrary $x = \sum_{k,l \in \overbar{\mathrm{B}}_\Lambda^\Z} t_{k-l} E_{kl} \in \mathrm{C}(\T^d)^{(\Lambda)}$ to yield
	$$
	U \left( ( \S_{\mathfrak{w}_\Lambda}   \otimes \id_{\mathcal{B}(\mathcal{H})} )\big([D,x] \otimes  \mathbf{1}_\H \big) \right)U^* = (\id_{\mathcal{B}(\mathcal{H})} \otimes \F_{\mathfrak{w}_\Lambda}  ) \left( U ([D, x] \otimes \mathbf{1}_\H )U^* \right).
	$$
	From this we obtain the following estimate:
	\begin{align}
		\| \S_{\mathfrak{w}_\Lambda}([D,x]) \| 
		&= \| \S_{\mathfrak{w}_\Lambda}([D,x]) \otimes \mathbf{1}_\mathcal{H} \| \\
		&= \| (\S_{\mathfrak{w}_\Lambda} \otimes \mathrm{id}_{\mathcal{B}(\mathcal{H})}) ([D,x] \otimes \mathbf{1}_\mathcal{H}) \| \\
		&= \| U \left( ( \S_{\mathfrak{w}_\Lambda}   \otimes \mathrm{id}_{\mathcal{B}(\mathcal{H})} )\big([D,x] \otimes  \mathbf{1}_\H \big) \right)U^* \| \\
		&= \| (\mathrm{id}_{\mathcal{B}(\mathcal{H})} \otimes \F_{\mathfrak{w}_\Lambda}  ) \left( U ([D, x] \otimes \mathbf{1}_\H )U^* \right) \| \\
		&\leq \| \mathrm{id}_{\mathcal{K}(\mathcal{H})} \otimes \F_{\mathfrak{w}_\Lambda} \| \| U ([D, x] \otimes \mathbf{1}_\H )U^* \| \\
		&= \| \F_{\mathfrak{w}_\Lambda}   \|_\cb \|[D,x]\|,
	\end{align}
	where the penultimate step follows from the fact that $U ([D, x] \otimes \mathbf{1}_\H )U^* \in \mathcal{K}(\mathcal{H}) \otimes [D,\mathrm{C}^\infty(\mathbb{T}^d)]$.
	This implies that $\| \S_{\mathfrak{w}_\Lambda}   \| \leq \| \F_{\mathfrak{w}_\Lambda}   \|_\cb$. 
	
	Finally, supposing that $\mathcal{F}_{{\mathfrak{w}_\Lambda} }$ is a bounded linear map its norm and $\mathrm{cb}$-norm coincide because its range is a subset of a commutative $\mathrm{C}^\ast$-algebra, namely $\mathrm{C}(\mathbb{T}^d)$ ({\em cf.} \cite[Theorem 3.9]{Pau02}). 
	\endproof
	
	\begin{remark}
		Note that the classical transference theorem is stated as an equality of the cb-norms of a Schur multiplier $\mathcal{S}_\varphi$ and a Fourier multiplier $\mathcal{F}_\varphi$.
		However, for this it is crucial that $\mathcal{F}_\varphi$ and $\mathcal{S}_\varphi$ are defined on the $\mathrm{C}^*$-algebras $\mathrm{C}_{\lambda}^*(\Gamma)$ and $\mathcal{B}(\ell^2(\Gamma))$ which is not the situation we find in the above lemma.
		In fact, the maps (\ref{eq:Fw}) and (\ref{eq:Sw}) are only defined on operator subsystems and do not extend to the $\mathrm{C}^*$-algebras $\mathrm{C}(\mathbb{T}^d)$ and $\mathcal{B}(\mathcal{H})$ respectively, so equality in \autoref{lem:Sw-Fw-norms} is not to be expected.
	\end{remark}
	
	Our task is thus reduced to the computation of the norm of the map $\mathcal{F}_{\mathfrak{w}_\Lambda} : [D,\mathrm{C}^\infty(\T^d)] \to \mathrm{C}^\infty(\T^d)$ given in Equation \eqref{eq:Fw}.

	\subsection{The spectral Fejér kernel}
	
	We define $K_{\mathfrak{m}_\Lambda}$ as the convolution kernel corresponding to the Fourier multiplier $\mathcal{F}_{\mathfrak{m}_\Lambda}$ as in (\ref{eqn:FourierMultiplierOnAbelianGroup}), i.e.
	\begin{align}
		K_{\mathfrak{m}_\Lambda} := \widecheck{\mathfrak{m}}_\Lambda.
	\end{align}
	In view of \autoref{lem:GoodKernelsGiveDiracEstimate} it would be desirable to see that $K_{\mathfrak{m}_\Lambda}$ is a \emph{good kernel}.
	Indeed, by \autoref{lem:CompositionsInTermsOfCommutators} this would give precisely the estimate of $\mathcal{F}_{\mathfrak{w}_\Lambda}$ which remains to show.
	
	\begin{lemma}\label{lem:PointwiseConvergenceOfSymbol}
		For every $n \in \mathbb{Z}^d$, we have that $\mathfrak{m}_\Lambda(n) \rightarrow 1$, as $\Lambda \rightarrow \infty$.
	\end{lemma}
	
	\begin{proof}
		This follows from two simple geometric observations.
		One is that $|\mathcal{N}_\mathrm{B}(\Lambda) - \mathcal{V}_\mathrm{B}(\Lambda)| \leq \sqrt{d} \mathcal{A}_\mathrm{B}(\Lambda)$, where $\mathcal{V}_\mathrm{B}(\Lambda)$ is the volume of the $d$-dimensional ball of radius $\Lambda$ and $\mathcal{A}_\mathrm{B}(\Lambda)$ is its surface area.
		The other is that a lense (i.e.\@ the intersection of two $d$-dimensional balls of radius $\Lambda$ one of them shifted by a parameter $n \in \mathbb{Z}^d$) contains the ball of radius $\Lambda-\|n\|$, if this number is non-negative.
		Together, these observations yield the following estimates:
		\begin{align}
			|1-\mathfrak{m}_\Lambda(n)|
			&= \frac{1}{\mathcal{N}_\mathrm{B}(\Lambda)} |\mathcal{N}_\mathrm{B}(\Lambda)-\mathcal{N}_\mathrm{L}(\Lambda,n)| \\
			&\leq \frac{1}{\mathcal{V}_\mathrm{B}(\Lambda - \sqrt{d})} |\mathcal{V}_\mathrm{B}(\Lambda+\sqrt{d}) - \mathcal{V}_\mathrm{B}((\Lambda-\sqrt{d}-\|n\|)_+)| \\
			&= \frac{1}{(\Lambda-\sqrt{d})^d} \underbrace{| (\Lambda+\sqrt{d})^d - ((\Lambda-\sqrt{d}-\|n\|)_+)^d |}_{= \mathcal{O}(\Lambda^{d-1})} \\
			&\rightarrow 0, 
		\end{align}
		as $\Lambda \rightarrow \infty$.
		(Here $t_+ := t$, if $t \geq 0$, and $t_+ := 0$, if $t < 0$, for $t \in \mathbb{R}$.)
	\end{proof}
	
	\begin{proposition}\label{prop:SpectralFejerKernel}
		The function $K_{\mathfrak{m}_\Lambda}$ is positive and is a \emph{good kernel}.
	\end{proposition}
	
	\begin{proof}
		We begin by showing positivity of $K_{\mathfrak{m}_\Lambda}$.
		Indeed, $K_{\mathfrak{m}_\Lambda} = \widecheck{\mathfrak{m}}_\Lambda$ is the (inverse) Fourier transform of a convolution-square: 
		\begin{align}\label{eqn:ConvolutionSquare}
			\begin{split}
				\mathfrak{m}_\Lambda(n)
				&= \frac{\mathcal{N}_\mathrm{L}(\Lambda,n)}{\mathcal{N}_\mathrm{B}} \\ 
				&= \frac{1}{\mathcal{N}_\mathrm{B}(\Lambda)} \sum_{k \in \mathbb{Z}^d} \chi_{\overbar{{\mathrm{B}}}_\Lambda^\mathbb{Z}}(k) \chi_{\overbar{{\mathrm{B}}}_\Lambda^\mathbb{Z}(n)}(k) \\
				&= \frac{1}{\mathcal{N}_\mathrm{B}(\Lambda)} \left(\chi_{\overbar{{\mathrm{B}}}_\Lambda^\mathbb{Z}} \ast \chi_{\overbar{{\mathrm{B}}}_\Lambda^\mathbb{Z}}\right) (n)
			\end{split}
		\end{align}
		
		Next, we check the total mass of $K_{\mathfrak{m}_\Lambda}$:
		\begin{align}
			\int_{\mathbb{T}^d} K_{\mathfrak{m}_\Lambda}(x) \, \mathrm{d}x
			&= \frac{1}{\mathcal{N}_\mathrm{B}(\Lambda)} \left(\chi_{\overbar{{\mathrm{B}}}_\Lambda^\mathbb{Z}} \ast \chi_{\overbar{{\mathrm{B}}}_\Lambda^\mathbb{Z}}\right) (0) \\
			&= \frac{\mathcal{N}_\mathrm{L}(\Lambda,0)}{\mathcal{N}_\mathrm{B}(\Lambda)}
			= 1
		\end{align}
		
		Last, we argue that the mass of $K_{\mathfrak{m}_\Lambda}$ becomes concentrated around $0$ as $\Lambda \rightarrow \infty$.
		Fix some $\delta > 0$ and let $\varepsilon > 0$ be arbitrary.
		Let $\varphi$ be a non-negative trigonometric polynomial such that the following holds:
		\begin{align}
			1- \chi_{\mathrm{B}_\delta}  \leq \varphi \leq 1 - \chi_{\mathrm{B}_{\nicefrac{\delta}{2}}} + \frac{\varepsilon}{2}
		\end{align}
		The existence of such a trigonometric polynomial $\varphi$ follows from Weierstraß's approximation theorem for trigonometric polynomials since the $\frac{\varepsilon}{4}$-neighborhood of the continuous function $\psi$, given by 
		\begin{align}
			\psi(x) = \begin{cases}
				\frac{\varepsilon}{4}, \text{ if } x \in \mathrm{B}_{\nicefrac{\delta}{2}}, \\
				\frac{2}{\delta} \|x\| + \frac{\varepsilon}{4} - 1, \text{ if } x \in \mathrm{B}_\delta \setminus \mathrm{B}_{\nicefrac{\delta}{2}}, \\
				1+\frac{\varepsilon}{4}, \text{ if } \mathrm{T}^d \setminus \mathrm{B}_\delta,
			\end{cases}
		\end{align}
		lies between $1- \chi_{\mathrm{B}_\delta}$ and $1 - \chi_{\mathrm{B}_{\nicefrac{\delta}{2}}} + \frac{\varepsilon}{2}$.
		Let $\Lambda_0 > 0$ be large enough such that, for all $\Lambda \geq \Lambda_0$, we have that $\displaystyle{\max_{n \in \mathrm{supp}(\widehat{\varphi})}} |1-\mathfrak{m}_\Lambda(n)| \leq \frac{\varepsilon}{2\|\widehat{\varphi}\|_{\ell^1(\mathbb{Z}^d)}}$.
		This is of course possible since the pointwise convergence from \autoref{lem:PointwiseConvergenceOfSymbol} implies uniform convergence to $1$ of the restriction of $\mathfrak{m}_\Lambda$ to the finite set $\mathrm{supp}(\widehat{\varphi})$.
		Then the following holds:
		\begin{align}
			\int_{\mathbb{T}^d \setminus \mathrm{B}_\delta} K_{\mathfrak{m}_\Lambda}(x) \, \mathrm{d}x
			&= \int_{\mathbb{T}^d} K_{\mathfrak{m}_\Lambda}(x)(1-\chi_{\mathrm{B}_\delta}(x)) \, \mathrm{d}x \\
			&\leq \int_{\mathbb{T}^d} K_{\mathfrak{m}_\Lambda}(x) \varphi(x) \, \mathrm{d}x \\
			&\leq \left|\int_{\mathbb{T}^d} K_{\mathfrak{m}_\Lambda}(x) \varphi(x) \, \mathrm{d}x - \varphi(0)\right| + \frac{\varepsilon}{2} \\
			&= \left|\sum_{n \in \mathbb{Z}^d} \mathfrak{m}_\Lambda(n) \widehat{\varphi}(n) - \sum_{n \in \mathbb{Z}^d} \widehat{\varphi}(n) \right| + \frac{\varepsilon}{2} \\
			&\leq \|\left(\mathfrak{m}_\Lambda-1\right)|_{\mathrm{supp}(\widehat{\varphi})}\|_{\ell^\infty(\mathbb{Z}^d)} \cdot \|\widehat{\varphi}\|_{\ell^1(\mathbb{Z}^d)} + \frac{\varepsilon}{2}\\
			&\leq \frac{\varepsilon}{2} + \frac{\varepsilon}{2},
		\end{align}
		for all $\Lambda \geq \Lambda_0$.
		In the fourth step we applied the Plancherel formula and in the fifth step the Hölder inequality.
	\end{proof}
	
	\begin{remark}
		Note that (\ref{eqn:ConvolutionSquare}) shows that the function $K_{\mathfrak{m}_\Lambda}$ is precisely the square of the well-known spherical Dirichlet kernel (\emph{cf.} Lipschitz seminorminition 3.1.6]{Gra14}).
		However, by a classical result by du Bois-Reymond the latter is not a \emph{good kernel} (\emph{cf.} \cite[Proposition 3.3.5]{Gra14}).
		The feature of $K_{\mathfrak{m}_\Lambda}$ which is crucial for its good behavior is positivity.
		We emphasize that our \emph{spectral Fejér kernel} does not coincide with the so-called \emph{circular Fejér kernel} which is investigated in \cite[Chapter 3]{Gra14}.
		Note in particular, that the \emph{circular} Fejér kernel fails to be \emph{good} in dimensions $d \geq 3$ which motivates the introduction of Bochner--Riesz summability methods.
	\end{remark}
	
	\begin{theorem}\label{thm:SpectralTruncationsConverge}
		Spectral truncations of $\mathbb{T}^d$ converge, for all $d \geq 1$.
	\end{theorem}
	
	\begin{proof}
		From $(T - \rho_\Lambda \circ \sigma_\Lambda(T)) \otimes \mathbf{1} 
		= \mathcal{S}_{\mathfrak{w}_\Lambda}(T)$ (\autoref{lem:CompositionsInTermsOfCommutators}) and $\|\mathcal{	S}_{\mathfrak{w}_\Lambda}\| \leq \|\mathcal{F}_{\mathfrak{w}_\Lambda}\|$ (\autoref{lem:Sw-Fw-norms}), together with $\mathcal{F}_{\mathfrak{w}_\Lambda}(f) = (f - \sigma_\Lambda \circ \rho_\Lambda(f)) \otimes \mathbf{1} = (f - K_{\mathfrak{w}_\Lambda} \ast f) \otimes \mathbf{1}$ (again \autoref{lem:CompositionsInTermsOfCommutators}), we obtain from \autoref{lem:GoodKernelsGiveDiracEstimate} a sequence $\gamma_\Lambda \rightarrow 0$, as $\Lambda \rightarrow \infty$, such that 
		\begin{align}
			\|T - \rho_\Lambda \circ \sigma_\Lambda(T)\| &\leq \gamma_\Lambda \|[D_\Lambda,T]\| \quad \text{and} \\
			\|f - \sigma_\Lambda \circ \rho_\Lambda(f)\| &\leq \gamma_\Lambda \|[D,f]\|,
		\end{align}
		since the spectral Fejér kernel $K_{\mathfrak{m}_\Lambda}$ is a good kernel (\autoref{prop:SpectralFejerKernel}).
		Together with \autoref{lem:PropertiesOfTheCompressionMap} and \autoref{lem:PropertiesOfTheUpwardsMap} this shows that $(\rho_\Lambda,\sigma_\Lambda)$ is a $\mathrm{C}^1$-approximate order isomorphism which by \cite[Theorem 5]{vS21} implies our result.
	\end{proof}
	
	\begin{remark}\label{rmk:OtherTruncationsConvergeToo}
		We point out that similarly convergence of other kinds of truncations of $\mathbb{T}^d$ can be shown.
		In particular, by replacing our projections $P_\Lambda$ with the projections $P_N^\square$ with $\mathrm{ran}(P_N^\square) = \{e_n \, : \, n_i = -N,\dots,N, \text{ for } i = 1,\dots,d\}$, analogous arguments to the ones presented in this chapter yield a new proof that the ``box-truncations'' of $\mathbb{T}^d$, which were considered in \cite{Ber19}, converge.
	\end{remark}

	\section{Structure Analysis of the Operator System $\mathrm{C}(\mathbb{T}^d)^{(\Lambda)}$}\label{sec:StructureAnalysis}
	
	\subsection{C*-envelope and propagation number}
	
	Recall \cite{Ham79} (see also \cite[Chapter 15]{Pau02}) that a \emph{$\mathrm{C}^*$-extension} of a unital operator system ${E}$ is a unital $\mathrm{C}^*$-algebra ${A}$ together with an injective completely positive map $\iota : {E} \rightarrow {A}$ such that $\mathrm{C}^*(\iota({E})) = {A}$.
	A $\mathrm{C}^*$-extension ${A}$ of ${E}$ is called the \emph{$\mathrm{C}^*$-envelope} and denoted by $\mathrm{C}_\mathrm{env}^*({E})$ if, for every unital $\mathrm{C}^*$-algebra ${B}$ and every unital completely positive map $\phi : {A} \rightarrow {B}$, the map $\phi$ is a complete order injection if the composition $\phi \circ \iota$ is.
	Recall furthermore from \cite{CvS21} that the \emph{propagation number} $\mathrm{prop}({E})$ is the smallest positive integer $n$ such that $(\iota({E}))^{\circ n} \subseteq \mathrm{C}_\mathrm{env}^*({E})$ is a $\mathrm{C}^*$-algebra, where ${F}^{\circ n} = \mathrm{span} \{ f_1 \cdots f_n \, : \, f_i \in {F}\}$, for a unital operator subsystem ${F}$ of a unital $\mathrm{C}^\ast$-algebra ${A}$.
	
	For $p \in \overbar{\mathrm{B}}_\Lambda^\mathbb{Z} + \overbar{\mathrm{B}}_\Lambda^\mathbb{Z}$, we define the following operator in $\mathrm{C}(\mathbb{T}^d)^{(\Lambda)} \subset \mathcal{B}(P_\Lambda \mathcal{H})$:
	\begin{align}\label{eqn:BasicToeplitzMatrix}
		T_p := \sum_{n \in \overbar{L}_\Lambda^{\mathbb{Z}}(p)} E_{n-p, n} 
		= \sum_{n \in \overbar{L}_\Lambda^{\mathbb{Z}}(-p)} E_{n,n+p},
	\end{align}
	where $E_{k,l} \in \mathcal{B}(P_\Lambda \mathcal{H})$ is the \emph{matrix unit} given by $\langle e_m , E_{k,l} e_n \rangle = \delta_{mk} \delta_{ln}$, for $k,l,m,n \in \overbar{\mathrm{B}}_\Lambda^\mathbb{Z}$.
	It is not hard to check that $\{T_p\}_{p \in \overbar{\mathrm{B}}_\Lambda^\mathbb{Z} + \overbar{\mathrm{B}}_\Lambda^\mathbb{Z}}$ is a basis for the operator system $\mathrm{C}(\mathbb{T}^d)^{(\Lambda)}$.
	
	With the preparations of \autoref{sec:ConvexGeometry}, we are in position to treat the $\mathrm{C}^*$-envelope and propagation number of the operator system $\mathrm{C}(\mathbb{T}^d)^{(\Lambda)}$.
	
	\begin{proposition}\label{prop:PropagationNumberOfSpectralMultiToeplitzSystem}
		The $\mathrm{C}^\ast$-envelope and the propagation number of $\mathrm{C}(\mathbb{T}^d)^{(\Lambda)}$ are given by  $\mathrm{C}^\ast_{\text{env}}(\mathrm{C}(\mathbb{T}^d)^{(\Lambda)}) = \mathcal{B}(P_\Lambda \mathcal{H})$ and $\prop(\mathrm{C}(\mathbb{T}^d)^{(\Lambda)}) = 2$.
	\end{proposition}
	
	\begin{proof}
		The matrix order structure on $\mathrm{C}(\mathbb{T}^d)^{(\Lambda)}$ is the one inherited from the inclusion into $\mathcal{B}(P_\Lambda \mathcal{H})$.
		It remains to show that the inclusion $\mathrm{C}(\mathbb{T}^d)^{(\Lambda)} \hookrightarrow \mathcal{B}(P_\Lambda \mathcal{H})$ is a $\mathrm{C}^\ast$-extension, i.e.\@ that it generates $\mathcal{B}(P_\Lambda \mathcal{H})$.
		Indeed, if this is the case, it is clear that $\mathcal{B}(P_\Lambda \mathcal{H}) \cong \mathrm{C}^\ast_{\text{env}}(\mathrm{C}(\mathbb{T}^d)^{(\Lambda)})$ since $\mathcal{B}(P_\Lambda \mathcal{H})$ is simple.
		
		We will see that $\mathcal{B}(P_\Lambda \mathcal{H})$ is in fact spanned by respective products of two basic operators \eqref{eqn:BasicToeplitzMatrix}.
		To this end, let $p, q \in \overbar{\mathrm{B}}_\Lambda^\mathbb{Z} + \overbar{\mathrm{B}}_\Lambda^\mathbb{Z}$.
		Then, the following holds:
		\begin{align}%\label{eqn:GeneralBasicToeplitzProducts}
			\begin{split}
				T_{p}T_{q}
				&= \left( \sum_{n \in \overbar{L}_\Lambda^{\mathbb{Z}}(p)} E_{n-p, n} \right) 
				\left( \sum_{n \in \overbar{L}_\Lambda^{\mathbb{Z}}(-q)} E_{n,n+q} \right) \\
				&= \sum_{n \in \overbar{L}_\Lambda^{\mathbb{Z}}(p) \cap \overbar{L}_\Lambda^{\mathbb{Z}}(-q)} E_{n-p,n+q} \\
				&= \sum_{n \in \overbar{L}_\Lambda^{\mathbb{Z}}(p+q) \cap \overbar{L}_\Lambda^{\mathbb{Z}}(q)} E_{n-p-q,n},
			\end{split}
		\end{align}
		where we used the fact that $\left(\overbar{L}_\Lambda(p) \cap \overbar{L}_\Lambda(-q)\right) + q = \overbar{L}_\Lambda(p+q) \cap \overbar{L}_\Lambda(q)$ which can be easily checked.
		As a special case, for $l, k \in \overbar{\mathrm{B}}_\Lambda^\mathbb{Z} + \overbar{\mathrm{B}}_\Lambda^\mathbb{Z}$ such that $l + k \in \overbar{\mathrm{B}}_\Lambda^\mathbb{Z} + \overbar{\mathrm{B}}_\Lambda^\mathbb{Z}$, we obtain:
		\begin{align}\label{eqn:SpecialBasicToeplitzProducts}
			T_{-k} T_{l+k}
			= \sum_{n \in \overbar{L}_\Lambda^{\mathbb{Z}}(l) \cap  \overbar{L}_\Lambda^{\mathbb{Z}}(l+k)} E_{n-l,n}
		\end{align}
		Note that this generalizes the formula given in the proof of \cite[Proposition 4.2]{CvS21} where $d$ was equal to $1$.
		\footnote{
			Moreover, this formula may be interpreted in a similar way:
			The operator $T_{-k} T_{l+k}$ can be regarded as ``matrix'' (with multi-indexed entries) which has $0$-entries everywhere except for the $l$-th ``diagonal'' where its entries are either $1$ or $0$ depending on the parameter $k$.
		}
		
		We need some elementary geometric observations.
		For $\Lambda' \geq 0$, let $K_{\Lambda'} := \co\left(\overline{\mathrm{B}}_{\Lambda'}^{\mathbb{Z}}\right)$ denote the convex hull of the set of $\mathbb{Z}^d$-lattice points in the closed ball of radius $\Lambda'$.
		Note that $K_{\Lambda'}^{\mathbb{Z}} := K_{\Lambda'} \cap \mathbb{Z}^d = \overline{\mathrm{B}}_{\Lambda'}^{\mathbb{Z}}$.
		Furthermore, $K_{\Lambda'}$ is a polytope which is symmetric under reflections along coordinate axes and diagonals (i.e.\@ under changing signs of coordinates and exchanging coordinates).
		Clearly, all the extreme points of $K_{\Lambda'}$, the set of which is denoted by $\ex (K_{\Lambda'})$, have integer coordinates.
		Moreover, if $x \in K_{\Lambda'}$ is of norm $\|x\| = \Lambda'$ it is an extreme point, but not necessarily all extreme points of $K_{\Lambda'}$ are of norm $\Lambda'$ as can be seen in the case $d=2$, $\Lambda' = 3$ (\emph{cf.} \autoref{fig:PlotPolytopesCircles}).
		
		\begin{figure}[h!]
			\centering
			\includegraphics[width=.90\linewidth]{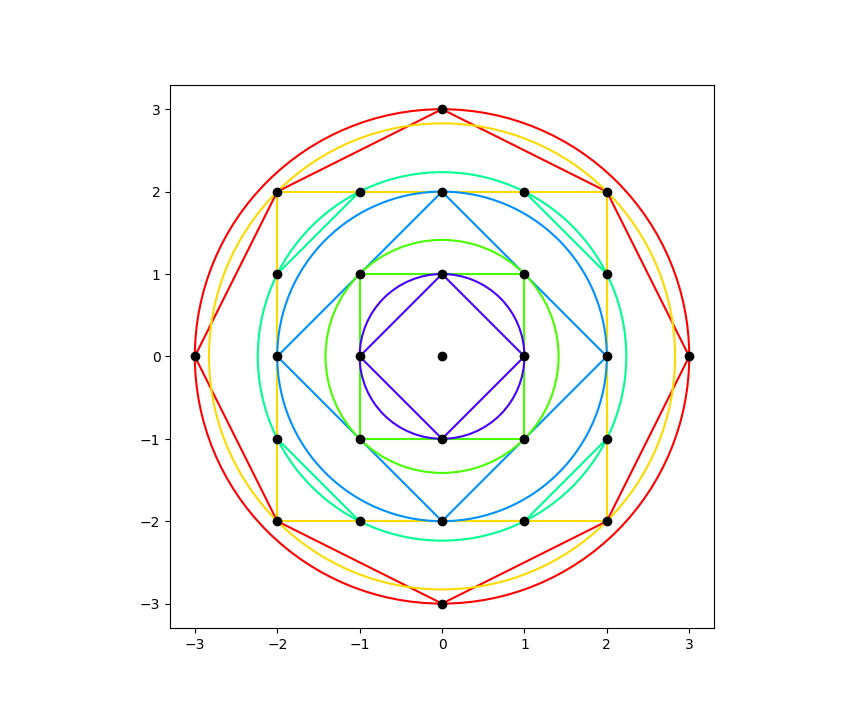}
			\caption{A plot of the boundaries of the balls $\mathrm{B}_\Lambda$ and the polytopes $K_\Lambda$, for $\Lambda = 0, 1, \sqrt{2}, 2, \sqrt{5}, 2\sqrt{2}, 3$.}
			\label{fig:PlotPolytopesCircles}
		\end{figure}
				
		In order to prove the claim it is enough to write every rank-one operator $E_{p,q} \in \mathcal{B}(P_\Lambda \mathcal{H})$ as a linear combination of products of the form (\ref{eqn:SpecialBasicToeplitzProducts}), where $p, q \in \overbar{\mathrm{B}}_\Lambda^\mathbb{Z}$ and $l, k \in \overbar{\mathrm{B}}_\Lambda^\mathbb{Z} + \overbar{\mathrm{B}}_\Lambda^\mathbb{Z}$ such that $l+k \in \overbar{\mathrm{B}}_\Lambda^\mathbb{Z} + \overbar{\mathrm{B}}_\Lambda^\mathbb{Z}$.
		To this end, fix $p, q \in \overbar{\mathrm{B}}_\Lambda^\mathbb{Z}$ and set $l := q-p \in \overbar{\mathrm{B}}_\Lambda^\mathbb{Z} + \overbar{\mathrm{B}}_\Lambda^\mathbb{Z}$.
		Set $\Lambda' := \|q\|$.
		We claim that we can find an extreme point $m \in \ex(K_\Lambda)$ such that the following holds:
		\begin{align}\label{eqn:ChoiceOfExtremalPoint}
			K_{\Lambda'} \cap (K_\Lambda - m + q) = \{q\}.
		\end{align}
		To see this, note that $q$ is an extreme point of $K_{\Lambda'}$ and that the smallest cone $[0,\infty) \cdot (K_{\Lambda'}-q) + q$ which contains $K_{\Lambda'}$ is locally compact (trivially in this finite dimensional case), closed, convex, proper and has vertex $q$. 
		By \autoref{lem:SeparateConvexCompactSetFromLocallyCompactClosedConvexSalientCone} we can find an extreme point $m \in \ex(K_\Lambda)$ such that $(K_\Lambda - m) \cap (K_{\Lambda'}-q) = (K_\Lambda - m) \cap \left([0,\infty) \cdot (K_{\Lambda'}-q) + q\right) = \{0\}$, which is equivalent to (\ref{eqn:ChoiceOfExtremalPoint}).
		
		Now, fix an extreme point $m \in \ex(K_{\Lambda})$ such that (\ref{eqn:ChoiceOfExtremalPoint}) is satisfied and set $k := q - l - m = p - m$.
		Note that for this $l$ and $k$ the product (\ref{eqn:SpecialBasicToeplitzProducts}) makes sense, i.e.\@ $k \in \overbar{\mathrm{B}}_\Lambda^\mathbb{Z} + \overbar{\mathrm{B}}_\Lambda^\mathbb{Z}$ and $l+k \in \overbar{\mathrm{B}}_\Lambda^\mathbb{Z} + \overbar{\mathrm{B}}_\Lambda^\mathbb{Z}$.
		Furthermore, the following holds:
		\begin{align}
			\left(\overbar{\mathrm{L}}_\Lambda^\mathbb{Z}(l) \cap \overbar{\mathrm{L}}_\Lambda^\mathbb{Z}(l+k)\right) \cap K_{\Lambda'} = \{q\}
		\end{align}
		The converse inclusion is clear from (\ref{eqn:ChoiceOfExtremalPoint}) together with $\overbar{\mathrm{L}}_\Lambda^\mathbb{Z}(l+k) \subset \overbar{\mathrm{B}}_\Lambda^\mathbb{Z}(l+k) = K_\Lambda^\mathbb{Z} + q-m$, where $q-m = l+k$ was used.
		
		This shows that, for $l = q-p$ and $k = p - m$, the rank-one operator $E_{p,q}$ is a summand of $T_{-k} T_{l+k}$ as in (\ref{eqn:SpecialBasicToeplitzProducts})
		where, for all the other rank-one operators $E_{n-l,n}$ appearing in the sum, we have that $\|n\| > \|q\|$, i.e.
		\begin{align}\label{eqn:SpecialBasicToeplitzProductsPeelingOffShells}
			E_{p,q} = T_{-k} T_{l+k} - \sum_{\underset{\|n\| > \|q\|}{n \in \overbar{\mathrm{L}}_\Lambda^\mathbb{Z}(l+k) \cap  \overbar{\mathrm{L}}_\Lambda^\mathbb{Z}(l)}} E_{n-l,n}.
		\end{align}
		Moreover, for each $E_{n-l,n}$ in the above sum, a similar expression can be obtained, and so forth.
		Hence, after finitely many steps this gives a finite linear combination of products of the form (\ref{eqn:SpecialBasicToeplitzProducts}) for $E_{p,q}$.
		
		Altogether, this proves that $\mathcal{B}(P_\Lambda \mathcal{H}) \subseteq \vsspan \left\{ T_{-k} T_{l+k} \, \big| \, l, k, l+k \in \overbar{\mathrm{B}}_\Lambda^\mathbb{Z} + \overbar{\mathrm{B}}_\Lambda^\mathbb{Z} \right\}$ which shows that $\mathrm{C}^\ast_{\text{env}}(\mathrm{C}(\mathbb{T}^d)^{(\Lambda)}) \cong \mathcal{B}(P_\Lambda \mathcal{H})$ and $\prop(\mathrm{C}(\mathbb{T}^d)^{(\Lambda)}) \leq 2$.
		Realizing that $E_{0,0} \notin \mathrm{C}(\mathbb{T}^d)^{(\Lambda)}$ it is clear that $\prop(\mathrm{C}(\mathbb{T}^d)^{(\Lambda)}) > 1$.
		This finishes the proof.
	\end{proof}
	
	We illustrate the procedure of expressing elementary matrices $E_{p,q} \in \mathcal{B}(P_\Lambda \mathcal{H})$ in terms of products of basic operators of the form \eqref{eqn:SpecialBasicToeplitzProducts} as described in the above proof in the following two examples.
	
	\begin{example}\label{ex:ProductOfBasicToeplitzMatricesForExtremePointsOfMaximalNorm}
		Let $p, q \in \overbar{\mathrm{B}}_\Lambda^\mathbb{Z}$ such that $\|q\| = \Lambda$.
		Finding an $m \in \ex(K_\Lambda)$ such that (\ref{eqn:ChoiceOfExtremalPoint}) holds is particularly easy in this case, namely, set $m := -q$.
		Then set $k := p-m = p + q$ and, with $l = q-p$, we obtain:
		\begin{align}
			E_{p,q} 
			= T_{-k} T_{l+k}
			= T_{-p-q} T_{q-p+p+q}
			= T_{-p-q} T_{2q},
		\end{align}
		according to (\ref{eqn:SpecialBasicToeplitzProductsPeelingOffShells}), since there are no $n \in \overbar{\mathrm{B}}_\Lambda^\mathbb{Z}$ with $\|n\| > \|q\| = \Lambda$.
	\end{example}
	
	\begin{example}
		Let $d=2$ and $\Lambda = \sqrt{2}$, i.e.\@ $\overbar{\mathrm{B}}_\Lambda^\mathbb{Z}$ consists of $9$ points and $K_\Lambda = \co\left(\overbar{\mathrm{B}}_\Lambda^\mathbb{Z}\right)$ is the square with side length $2$.
		For $p = (0,0)$ and $q = (1,0)$, we want to express the matrix unit $E_{p,q}$ as a linear combination of products of basic operators \eqref{eqn:SpecialBasicToeplitzProducts} with $l = q-p = (1,0)$.
		Set $\Lambda' := \|q\| = 1$.
		Now, find an extreme point $m \in \ex(K_\Lambda)$ such that (\ref{eqn:ChoiceOfExtremalPoint}) holds.
		A valid choice is e.g.\@ $m := (-1,-1)$.
		Set $k := p-m = (1,1)$.
		Then we have:
		\[
		\overbar{\mathrm{L}}_\Lambda^\mathbb{Z}(l) \cap \overbar{\mathrm{L}}_\Lambda^\mathbb{Z}(l+k)
		= \overbar{\mathrm{L}}_{\sqrt{2}}^\mathbb{Z}(1,0) \cap \overbar{\mathrm{L}}_{\sqrt{2}}^\mathbb{Z}(2,1)
		= \{(1,0), (1,1)\}
		\]
		Therefore:
		\[
		T_{-k}T_{l+k}
		= T_{(-1,-1)}T_{(2,1)}
		= \sum_{n \in \overbar{\mathrm{L}}_{\sqrt{2}}^\mathbb{Z}(1,0) \cap \overbar{\mathrm{L}}_{\sqrt{2}}^\mathbb{Z}(2,1)} E_{n-(1,0),n}
		= \underbrace{E_{(0,0),(1,0)}}_{= E_{p,q}} + E_{(0,1),(1,1)}
		\]
		Note that $\|(1,1)\| = \sqrt{2} > \|q\| = 1$.
		
		By \autoref{ex:ProductOfBasicToeplitzMatricesForExtremePointsOfMaximalNorm}, we have $E_{(0,1),(1,1)} = T_{(-1,-2)}T_{(2,2)}$.
		Altogether, we obtain:
		\[
		E_{(0,0),(1,0)} = T_{(-1,-1)}T_{(2,1)} - T_{(-1,-2)}T_{(2,2)}
		\]
	\end{example}
	
	\begin{remark}
		We point out that analogously to the proof of \autoref{prop:PropagationNumberOfSpectralMultiToeplitzSystem} one can show that the $\mathrm{C}^*_\mathrm{env}(P_\Lambda^K \mathrm{C}^\infty(\mathbb{T}^d) P_\Lambda^K) = \mathcal{B}(P_\Lambda^K \mathcal{H})$ and $\mathrm{prop}(P_\Lambda^K \mathrm{C}^\infty(\mathbb{T}^d) P_\Lambda^K) = 2$, where $K$ is a convex compact subset of $\mathbb{R}^d$ which is symmetric with respect to reflections along the coordinate axes and diagonals and where $P_\Lambda^K \in \mathcal{B}(\mathcal{H})$ is the orthogonal projection with $\mathrm{ran}(P_\Lambda^K) = \{e_n \, : \, n \in (\Lambda \cdot K) \cap \mathbb{Z}^d\}$.
		In particular, this shows that the propagation number of the operator system obtained from ``box-truncations'', which were considered in \cite{Ber19}, is also $2$.
	\end{remark}

	\subsection{Dual}
	
	In \cite[Theorem 3.1]{Far21} it was shown that the operator system of $(N \times \nolinebreak N)$-Toeplitz matrices $\mathrm{C}(S^1)^{(N)}$ is dual to the Fejér--Riesz system $\mathrm{C}(S^1)_{(N)}$ which consists of trigonometric polynomials of the form $\sum_{n = -N+1}^{N-1} a_n e_n$. 
	In fact, this was already stated in \cite{CvS21} but it was only shown in \cite{Far21} that $\mathrm{C}(S^1)^{(N)} \cong \left(\mathrm{C}(S^1)_{(N)}\right)^{\mathrm{d}}$ in the sense that there is a unital \textit{complete} order isomorphism.
	For the proof, the (operator valued) Fejér--Riesz theorem plays an essential role.
	Recall that the Fejér--Riesz theorem states that every non-negative Laurent polynomial $P = \sum_{k = -N}^N a_k e_k \geq 0$ can be expressed as a hermitian square of an analytic polynomial $Q = \sum_{k=0}^N b_k e_k$, i.e.\@
	\[ P = Q^*Q.\]
	A generalization to the case where the coefficients $a_k$ and $b_k$ are operators on a Hilbert space is due to Rosenblum.
	See \cite{DR10} for a survey on the operator-valued Fejér--Riesz theorem.
	
	In view of the duality result for $S^1$, it would be natural to expect that the operator system $\mathrm{C}(\mathbb{T}^d)^{(\Lambda)}$ is dual to the operator system $\mathrm{C}(\mathbb{T}^d)_{(2\Lambda)}$ which consists of trigonometric polynomials on the $d$-torus of the form $\sum_{n \in \overbar{\mathrm{B}}_{2\Lambda}^\mathbb{Z}} a_n e_n$.
	However, this duality must fail even algebraically as soon as $d \geq 2$.
	Indeed, it is clear that $\dim\left(\mathrm{C}(\mathbb{T}^d)_{(2\Lambda)}\right) = \# \overbar{\mathrm{B}}_{2\Lambda}^\mathbb{Z} = \mathcal{N}_\mathrm{B}(2\Lambda)$ and from the above considerations about a basis we conclude that $\dim\left(\mathrm{C}(\mathbb{T}^d)^{(\Lambda)}\right) = \# \left(\overbar{\mathrm{B}}_\Lambda^\mathbb{Z} - \overbar{\mathrm{B}}_\Lambda^\mathbb{Z}\right) = \# \left(\overbar{\mathrm{B}}_\Lambda^\mathbb{Z} + \overbar{\mathrm{B}}_\Lambda^\mathbb{Z}\right)$.
	In general, however, the inclusion $\overbar{\mathrm{B}}_\Lambda^\mathbb{Z} + \overbar{\mathrm{B}}_\Lambda^\mathbb{Z} \subset\overbar{\mathrm{B}}_{2\Lambda}^\mathbb{Z}$ is strict, as the following examples demonstrate:
	\begin{example}
		If $d=2$, we have that $(3,2) \in \overbar{\mathrm{B}}_4^\mathbb{Z} \setminus \left(\overbar{\mathrm{B}}_2^\mathbb{Z} + \overbar{\mathrm{B}}_2^\mathbb{Z}\right)$.
		If $d \geq 3$, we have that $(1,\cdots,1) \in \overbar{\mathrm{B}}_2^\mathbb{Z} \setminus \left(\overbar{\mathrm{B}}_1^\mathbb{Z} + \overbar{\mathrm{B}}_1^\mathbb{Z}\right)$.
	\end{example}
	In view of these remarks, a more promising candidate for the dual operator system of $\mathrm{C}(\mathbb{T}^d)^{(\Lambda)}$ might be the operator subsystem $\mathrm{C}(\mathbb{T}^d)_{(\Lambda) + (\Lambda)}$  of $\mathrm{C}(\mathbb{T}^d)$ which consists of trigonometric polynomials on the $d$-torus of the form $\sum_{n \in \overbar{\mathrm{B}}_{\Lambda}^\mathbb{Z} + \overbar{\mathrm{B}}_{\Lambda}^\mathbb{Z}} a_n e_n$.
	
	In order to generalize the proof of \cite{Far21} to this setting, 
	one would need to show that the following map is a complete order isomorphism:
	\begin{align}
		\phi : \mathrm{C}(\mathbb{T}^d)^{(\Lambda)} &\rightarrow (\mathrm{C}(\mathbb{T}^d)_{(\Lambda)+(\Lambda)})^\mathrm{d} \\
		\mathbf{t} &\mapsto \varphi_\mathbf{t},
	\end{align}
	where 
	\begin{align}
		\varphi_\mathbf{t}(f) = \sum_{n \in \overbar{\mathrm{B}}_\Lambda^\mathbb{Z} + \overbar{\mathrm{B}}_\Lambda^\mathbb{Z}} t_{-n} a_n,
	\end{align}
	for $\mathbf{t} = (t_{k-l})_{k,l \in \overbar{\mathrm{B}}_\Lambda^\mathbb{Z} +\overbar{\mathrm{B}}_\Lambda^\mathbb{Z}}$ and $f = \sum_{n \in \overbar{\mathrm{B}}_\Lambda^\mathbb{Z} + \overbar{\mathrm{B}}_\Lambda^\mathbb{Z}} a_n e_n \in \mathrm{C}(\mathbb{T}^d)_{(\Lambda)+(\Lambda)}$.
	It is easy to see that $\phi$ is unital and an isomorphism of vector spaces, so it remains to show that $\phi$ and $\phi^{-1}$ are completely positive.
	
	While the proof of complete positivity of $\phi^{-1}$ is largely analogous to the one in \cite{Far21}, for the proof of complete positivity of $\phi$ a multivariate version of the (operator-valued) Fejér--Riesz theorem would be needed.
	In fact, it would be sufficient to be able to express a (dense subset of the cone of) non-negative Laurent polynomials in $d$ variables $P = \sum_{n \in \overbar{\mathrm{B}}_{\Lambda}^\mathbb{Z}} a_n e_n$ as a finite sum of squares of analytic polynomials $Q_i = \sum_{n \in \left(\overbar{\mathrm{B}}_{\Lambda}^\mathbb{Z}\right)_+} b_n^i e_n$, where $\left(\overbar{\mathrm{B}}_{\Lambda}^\mathbb{Z}\right)_+ = \overbar{\mathrm{B}}_\Lambda \cap (\mathbb{Z}_{\geq 0})^d$.
	Although it is not known to the authors if such a result holds for a dense subset of the positive cone
, at least for $d \geq 3$ it must fail for some non-negative Laurent polynomials even if higher degrees of the analytic polynomials $Q_i$ are admitted (\cite{Sch00}).
	For $d=2$, Dritschel proved (\cite[Theorem 4.1]{Dri18}) that if $P$ is a non-negative Laurent polynomial of degree $(d_1,d_2)$ (with Hilbert space operators as coefficients), then $P$ is a sum of at most $2d_2$ squares of analytic polynomials each of degree at most $(d_1,d_2-1)$.
	However, as it is stated in the conclusion of that article it is not clear by how much the number of summands and in particular the bound on the degree can be improved.
	Yet, there are examples of non-negative polynomials of degree $(d_1,d_2)$ which are not a finite sum of squares of analytic polynomials of degree $(d_1,d_2)$ (\cite{NS85}, \cite{Rud63}, \cite[Section 3.6]{Sak97}). 
	Note that it is in fact true that a strictly positive trigonometric polynomial $P > 0$ in $d$ variables is a finite sum of squares of analytic polynomials (\cite[Theorem 5.1]{DR10}), but the degrees might get out of control. 
	
	We see that it is apparently not as straightforward to determine the operator system dual of $\mathrm{C}(\mathbb{T}^d)^{(\Lambda)}$ as one might expect at first thought and we have to leave this to further research to be conducted.

	\appendix
	
	\section{Some convex geometry}\label{sec:ConvexGeometry}
	
	In order to compute the propagation number of the operator system $\mathrm{C}(\mathbb{T}^d)^{(\Lambda)}$, some facts from convex geometry are required.
	Since the natural setting for these is locally convex spaces we formulate all the required results in this abstract language even though we only make use of them in the finite dimensional case.
	See e.g.\@ \cite[Section 11]{Roc97} and \cite[Chapter II]{Bou87} for much of the standard terminology.
	
	Throughout this section, let $X$ be a Hausdorff locally convex space over $\mathbb{R}$ with continuous dual $X'$.
	Every linear functional $l$ on $X$ and every real number $\alpha \in \mathbb{R}$ give rise to a hyperplane in $X$ given by $H_{l = \alpha} =  \left\{x \in X \, | \, l(x) = \alpha\right\}$.
	Clearly, $H_{l=\alpha} = \ker(l-\alpha)$ and the hyperplane $H_{l=\alpha}$ is closed if and only if $l$ is continuous which is the only case we consider.
	\iffalse
	Clearly, $H_{l=\alpha} = \ker(l-\alpha)$ which shows that if $l$ is continuous, $H_{l=\alpha}$ is closed.
	The converse is also true (\cite[II, §2.2, Theorem 1]{Bou87}).
	Moreover, every hyperplane of $X$ is either closed or everywhere dense (\cite[I, §2.1, Corollary to Proposition 1]{Bou87}).
	In the sequel, we only consider closed hyperplanes, so it is always assumed that, for a hyperplane $H_{l=\alpha}$, the functional $l$ is in $X'$.
	\fi
	Every hyperplane $H_{l=\alpha}$ gives rise to an open positive and an open negative half-space denoted by $H_{l>\alpha}$ and $H_{l<\alpha}$ respectively and defined by $H_{l>\alpha} := \left\{x \in X \, | \, l(x) > \alpha\right\}$ and similarly for $H_{l<\alpha}$.
	Their respective closures are called the closed positive and the closed negative half-space associated to $H_{l=\alpha}$ and denoted by $H_{l\geq\alpha}$ and $H_{l\leq\alpha}$ respectively. 
	Of course $H_{l\geq\alpha} = \left\{x \in X \, | \, l(x) \geq \alpha\right\}$ and similarly for $H_{l\leq\alpha}$.
	
	If $K, L \subseteq X$ are two convex sets, it is said that they are \emph{separated} by the hyperplane $H_{l=\alpha}$ if $K \subseteq H_{l \geq \alpha}$ and $L \subseteq H_{l \leq \alpha}$.
	The sets $K$ and $L$ are called \emph{properly separated} if additionally $K \nsubset H_{l=\alpha}$ or $L \nsubset H_{l=\alpha}$.
	The sets $K$ and $L$ are called \emph{strictly separated} by $H_{l=\alpha}$ if $K \subseteq H_{l > \alpha}$ and $L \subseteq H_{l < \alpha}$.
	A hyperplane $H_{l=\alpha}$ is called a \emph{supporting hyperplane} for a non-empty convex set $K \subset X$, if $K \subseteq H_{l \geq \alpha}$ and $x \in H_{l=\alpha}$, for at least one $x \in K$.
	A supporting hyperplane $H_{l=\alpha}$ for $K$ is called \emph{non-trivial} if $K \nsubseteq H_{l=\alpha}$.
	
	Recall that a cone $C \subseteq X$ is called \emph{pointed} if $0 \in C$, \emph{salient} if it does not contain any $1$-dimensional subspaces of $X$ and \emph{proper} if $C \cap (-C) = \{0\}$.
	We only consider convex cones. 
	A convex pointed cone is salient if and only if it is proper.
	For $x \in X$, we call a set $C + x$ a \emph{cone with vertex $x$} if $C$ is a pointed cone.
	A cone with vertex $x$ is called \emph{proper} \iffalse (equivalently \emph{salient})\fi if $C-x$ is a proper cone.
	
	Let $K \subset X$ be convex set with $0 \notin K$.
	Set $C := [0,\infty) \cdot K$.
	Clearly, $C$ is a convex pointed cone.
	Moreover, $C$ is the smallest convex pointed cone which which contains $K$ in the sense that every convex pointed cone which contains $K$ must contain $C$.
	
	If $C$ is a convex pointed cone, a subset $B \subset C$ is called a \emph{base} (or \emph{sole} in \cite[II, §8.3]{Bou87}) if there exists a closed hyperplane $H \notni 0$ such that $B = H \cap C$ and such that $C$ is the smallest convex pointed cone which contains $B$.
	It is well-known that a convex subset $B$ of a convex pointed cone $C$ is a base if and only if, for every $x \in C \setminus \{0\}$, there exists a unique pair $(\lambda,y) \in (0,\infty) \times B$ such that $x = \lambda y$.
		
	The statement of the following lemma can be found in \cite[II, §7.2, Exercise 21a]{Bou87}.
	
	\begin{lemma}\label{lem:Bourbaki.Ex.II.2.21.a}
		Let $C \subset X$ be a locally compact, closed, convex, proper cone with vertex $x$.
		Then there exists a closed supporting hyperplane $H$ of $C$ such that $H \cap C = \{x\}$.
	\end{lemma}
	
	\begin{proof}
		To simplify notation we assume, without loss of generality, that $x = 0$.
		Let $U \subset X$ be a convex open neighborhood of $0$ such that $K := \overbar{U} \cap C$ is compact.
		We claim that $C = [0,\infty) \cdot K$, i.e.\@ $C$ is the smallest convex pointed cone which contains $K$.
		In fact, the inclusion $C \supseteq [0,\infty) \cdot K$ is clear from the cone property.
		To see that $C \subseteq [0,\infty) \cdot K$, let $y \in C$.
		Since every $0$-neighborhood in a locally convex space is absorbent, there exists a positive scalar $\lambda > 0$ such that $y \in \lambda U$.
		Hence, $\frac{1}{\lambda} y \in U \cap \frac{1}{\lambda} C = U \cap C \subset K$ and therefore $y = \lambda \frac{1}{\lambda} y \in \lambda K \subset [0,\infty) \cdot K$.
		
		By \cite[II, §7.1, Proposition 2]{Bou87}, there is an open half-space $H_{l < \alpha} \subset X$ such that $0 \in H_{l < \alpha} \cap K \subset U \cap C$.
		We may assume that the scalar $\alpha$ is positive, otherwise pass to the functional $-l$ instead.
		The boundary $\partial H_{l < \alpha} = H_{l = \alpha}$ of this half-space is a closed hyperplane of $X$ which does not contain $0$.
		
		We claim that the cone $C$ is the smallest closed convex pointed cone which contains $H_{l = \alpha} \cap C$, i.e.\@ $C = [0,\infty) \cdot (H_{l=\alpha} \cap C)$.
		To see this, observe that the following inclusion holds:
		\begin{align}%\label{eqn:TheConesOverIntersectionWithHyperplaneAndHalfSpaceAreEqual}
			[0,\infty) \cdot \left(H_{l = \alpha} \cap C\right) 
			= [0,\infty) \cdot \left(H_{l \leq \alpha} \cap C\right)
			\subseteq [0, \infty) \cdot K = C
		\end{align} 
		The converse inclusion $[0,\infty) \cdot (H_{l \leq \alpha} \cap C) \supseteq [0,\infty) \cdot K$ follows from the continuity and hence boundedness on the compact set $K$ of the functional $l$ by observing that there exists a positive real number $\lambda > 0$ such that $\lambda H_{l \leq \alpha} = H_{l \leq \lambda \alpha} \supseteq K$.
		
		Now, set $H := H_{l=0}$.
		Then we have:
		\begin{align}
			H \cap C = H \cap [0, \infty) \cdot (H_{l = \alpha} \cap C) = ( H \cap [0, \infty) \cdot H_{l=\alpha} ) \cap C = \{0\}
		\end{align}
		So $H$ is the desired closed supporting hyperplane for $C$.
	\end{proof}
	
	\begin{lemma}\label{lem:CpctConvSetsCanBeShiftedToPositiveHalfSpace}
		Let $K \subset X$ be a non-empty compact convex subset.
		Let $l \in X'$ be a continuous linear functional and $H_{l = 0}$ be the associated closed hyperplane through $0$.
		Then there exists an extreme point $x \in \ex(K)$ such that $K-x \subset H_{l \geq 0}$.
	\end{lemma}
	
	\begin{proof}
		By continuity of $l$, the image $l(K) \subset \mathbb{R}$ is bounded.
		Set $\alpha := \inf(l(K))$.
		In other words, $\alpha$ is the largest real number such that $K \subset H_{l \geq \alpha}$.
		In particular, $H_{l=\alpha}$ is a closed supporting hyperplane of $K$.
		By \cite[II, §7.1, Corollary to Proposition 1]{Bou87}, the hyperplane $H_{l=\alpha}$ contains an extreme point $x$ of $K$.
		Moreover, $K-x$ is contained in the positive half-space $H_{l \geq 0}$.
	\end{proof}
	
	Combining the previous two lemmas, we obtain the following immediate consequence:
	
	\begin{corollary}\label{lem:SeparateConvexCompactSetFromLocallyCompactClosedConvexSalientCone}
		Let $C \subset X$ be a locally compact, closed, convex, proper cone with vertex $x$ and let $K \subset X$ be a non-empty compact convex set.
		Then there exists an extreme point $y \in \ex(K)$ and a closed hyperplane $H$ which separates $C-x$ and $K-y$.
		Moreover, these two sets only intersect in $0$, i.e.\@ $(C-x) \cap (K-y) = \{0\}$. 
	\end{corollary}	
	
%	\bibliographystyle{plain}
%	\bibliography{resubmission240117}

\end{document}